\documentclass[preprint,authoryear]{elsarticle}
\usepackage[ansinew]{inputenc}
\usepackage{a4wide}
\usepackage{calc}
\usepackage{times}
\usepackage{diagbox}
\usepackage{amssymb,amsmath,amstext,amsthm,amsfonts, ams fonts}
\usepackage{colortbl,color,graphics,graphicx,epsfig}
\definecolor{gray}{rgb}{0.8,0.8,0.8}
\graphicspath{{pics/}}
\usepackage{dsfont}
\usepackage{bbm}
\usepackage{nicefrac}
\usepackage{lineno}
\usepackage[symbol]{footmisc}
\definecolor{dblue}{rgb}{0.21,0.21,0.55}

\global\long\def\intdiff{\mathrm{d}}
\usepackage[pdflinkmargin=5pt,pdfstartview={FitBH -32768},plainpages=false]{hyperref}

\renewcommand{\P}{\mathbb{P}}

\newcommand{\E}{\mathbb{E}}
\newcommand{\N}{\mathds{N}}

\newcommand{\1}{\mathbbm{1}}
\newcommand{\KLEINO}{{\scriptstyle{\mathcal{O}}}}

\DeclareMathAccent{\verywidehat}{\mathord}{largesymbols}{'144}
\newcommand{\var}{\mathbb{V}\hspace*{-0.05cm}\textnormal{a\hspace*{0.02cm}r}}

\newdefinition{remark}{Remark}
\newdefinition{defi}{Definition}
\newtheorem{theo}{Theorem}

\newtheorem{prop}{Proposition}[section]
\newtheorem{lem}[prop]{Lemma}
\newtheorem{cor}[prop]{Corollary}
\newtheorem{assumpsec}[prop]{Assumption}

\newcounter{assumpl}
\begin{document}
\begin{frontmatter}
\title{Testing for jumps in processes with integral fractional part and jump-robust inference on the Hurst exponent}
\author[1]{Markus Bibinger\footnotemark[1]
}
\author[1]{Michael Sonntag\footnote[1]{Financial support from the Deutsche Forschungsgemeinschaft (DFG) under grant 403176476 is gratefully acknowledged.}}
\address[1]{Faculty of Mathematics and Computer Science, Institute of Mathematics, Julius-Maximilians-Universit\"at W\"urzburg.}
\begin{abstract}
We develop and investigate a test for jumps based on high-frequency observations of a fractional process with an additive jump component. The Hurst exponent of the fractional process is unknown. The asymptotic theory under infill asymptotics builds upon extreme value theory for weakly dependent, stationary time series and extends techniques for the semimartingale case from the literature. It is shown that the statistic on which the test is based on weakly converges to a Gumbel distribution under the null hypothesis of no jumps. We prove consistency under the alternative hypothesis when there are jumps. Moreover, we establish convergence rates for local alternatives and consistent estimation of jump times. 
In the process, we show that inference on the Hurst exponent of a rough fractional process is robust with respect to jumps. This provides an important insight for the growing literature on rough volatility. We demonstrate sound finite-sample properties in a simulation study and showcase the applicability of our methods in an empirical example with a time series of volatilities.
\end{abstract}
\end{frontmatter}
\thispagestyle{plain}
\section{Introduction\label{sec:1}}
This work is devoted to the problem of testing for jumps in a discretely observed integral fractional process. In the literature on statistics for stochastic processes, there have been many contributions considering semimartingales with jumps. Another strand of research considers fractional processes without jumps, while models with jumps and fractional components did so far not attract much attention.

Since jump tests for semimartingales heavily exploit the known magnitude of increments of a continuous semimartingale, it is a statistically and mathematically interesting problem to consider jumps in a fractional process with unknown Hurst exponent when this is not possible. A practical motivation for this work is moreover the recent interest in fractional stochastic volatility models in mathematical finance, econometrics, and statistics, starting with \cite{roughvola}. As explained in \cite{chongrosenbaum} and in \cite{Wang}, over longer time scales the volatility process, either estimated through realized volatilities over disjoint subintervals from intradaily ultra-high-frequency data, or the logarithmic realized volatility itself, can be adequately modelled by discrete observations of a fractional process. In this framework, discrete observations of the volatility, which is usually latent, are available, either directly or with negligible noise from the pre-estimation based on ultra-high-frequency prices from different time intervals. Using these models, empirical evidence for rough fractional volatility has been consistently found by \cite{fukasawarough}, \cite{lunderough} and \cite{Wang}, among others. These models and inference methods obtained therein turned out to be empirically quite successful, for instance, in volatility forecasting, see, e.g.\ \cite{Wang} and \cite{forecast}. The literature so far considers continuous processes as underlying continuous-time volatility models. In view of the impact of news events on financial markets and empirical evidence for volatility jumps which has been documented, for instance, in \cite{voljumps} and \cite{bibwinkellev}, an additive mixture model with a fractional continuous component and additional jumps appears to be a natural extension. The importance of including jumps in the fractional model is emphasized as well in the introduction of \cite{chongrosenbaum} and the conclusion of \cite{Wang}. Therefore, we consider a stochastic process of the form
\[X_{t}=Y_{t}+J_{t}\,,\]
where $(Y_t)$ is a fractional process with continuous paths and general Hurst exponent $H\in(0,1)$, and $(J_{t})$ a general jump process.

We aim to separate the jumps from the fractional process with continuous paths based on discrete high-frequency observations over the fix time interval $[0,1]$. There is extensive literature on disentangling jumps and a continuous component of semimartingales based on high-frequency data. For this problem, as the distance between observation times $1/n\to 0$, the scaling $n^{-1/2}$ of absolute increments of a continuous semimartingale is exploited. Truncation methods pioneered in \cite{mancini} ascribe (much) larger increments to jumps and discard these increments for jump-robust volatility estimation. Combined with extreme value theory, \cite{leemykland} used the asymptotic Gumbel distribution of the rescaled maximal absolute increment, standardized with a spot volatility estimate, to establish a popular test for jumps further studied by \cite{palmes2016gumbel, palmes}. The rich literature on high-frequency statistics for semimartingales with jumps and truncation methods is partly summarized in the book by \cite{jacodprotter} and recent contributions include, among others, \cite{figueroa2019optimum}, \cite{amorino2020unbiased} and \cite{inatsugu2021global}.

In principle, truncation methods could be extended to a fractional continuous process with an additive jump component if the Hurst exponent $H$ was known. Then the scaling $n^{-H}$ of absolute increments of the continuous process could be exploited. Here, we are interested in the more difficult problem when $H$ is unknown. Intuitively, the comparison of observed large absolute increments and average absolute increments should still allow detecting jumps. Our idea is hence that the empirical distribution of increments reveals large absolute increments due to jumps separating from the remaining majority of increments.
 
Exploring methods based on this idea to disentangle jumps and continuous increments from a fractional process with unknown $H$, we found that the testing problem can be optimally solved using a statistic similar to the one by \cite{leemykland}. In particular, it is natural to construct a test based on the maximum of suitably normalised absolute increments. Since the maximal absolute increment and the normalizing spot volatility estimate contain the same scaling factors $n^{H}$ depending on $H$, they cancel out. Based on extreme value theory for weakly dependent Gaussian time series, we prove that under the null hypothesis of no jumps the limiting distribution is a standard Gumbel distribution. This yields a similar test as in the semimartingale case. We believe that it is appealing to show that the well-known procedure can be adapted to the fractional setting. A main difficulty, however, is that the spot volatility estimation should be robust with respect to jumps without knowing $H$. To achieve this, we use second-order increments and power variations with small powers what solves the problem under some restrictions on the jump activity. 

Related to this insight, we point out that a jump-robust estimation of the Hurst exponent is as well feasible. This answers a question which is currently of great interest in view of the empirical evidence for rough volatility, if ignoring jumps can manipulate the results. Our theoretical findings are good news for the existing literature, since we show that standard estimators of $H$ work asymptotically also in the presence of jumps which are ignored, if the true values of $H$ are small. The question if a potential influence of jumps is negligible in finite-sample applications is, however, more delicate. On the one hand, ignoring jumps for larger values of $H$ should not result in manipulated small estimates. In fact, our simulations show that even if a finite-sample bias due to jumps becomes relevant, it will be positive, resulting in larger estimates. On the other hand, additional to jumps we also find some volatility outliers in a data example and demonstrate that they induce a negative bias for the estimation of $H$.  

We derive a novel result about the estimation of jump times with a very fast rate of convergence. This allows the localization of jumps. It can also be used to filter out jumps before estimating $H$, what will further improve inference on rough fractional volatility models. We use this in an application to a time series of daily volatility estimates based on intra-daily ultra-high-frequency prices containing data analysed in \cite{forecast} with a parametric fractional Brownian motion model. Our methods locate several jumps and volatility outliers in the time series and we find that estimates of $H$ slightly increase after filtering them out.

This paper is organised as follows: In Section \ref{sec:2}, we introduce our observation model, the testing problem and fix some notation and assumptions. The construction of the statistics and tests is motivated and outlined in Section \ref{sec:3}. Section \ref{sec:4} highlights the jump-robust inference on the Hurst exponent $H$, which is related to required ingredients for the asymptotic theory of the test. In Section \ref{sec:5}, we present the statistical methods to test for jumps and the main results of this paper. They include the convergence of our normalised test statistic under the null hypothesis to a Gumbel limit distribution and the consistency under the alternative hypothesis. A simulation study analysing the finite-sample performance of the methods is summarized in Section \ref{sec:6}. Section \ref{sec:6neu} provides the application to real volatility data with a discussion of interesting stylized facts and the empirical insights. Section \ref{sec:7} concludes. All proofs are given in Section \ref{sec:8}. To make this work self-contained, we include some crucial non-standard prerequisites in the proofs section.

\section{Model, assumptions, and testing problem\label{sec:2}}
\subsection{Observation model\label{subsection 2.1}}
On some underlying probability space $(\Omega,\mathcal{A},\mathbb{P})$, let $B^H=(B_t^H)_{t\in[0,1]}$ be a fractional Brownian motion with Hurst exponent $H\in (0,1)$, that is a centred Gaussian process with continuous paths and covariance function
\begin{eqnarray*}
\mathbb{E}[B_t^HB_s^H]=\frac{1}{2}\left( t^{2H}+s^{2H}-\vert t-s \vert^{2H} \right),\quad t,s\in[0,1].
\end{eqnarray*}
With the Kolmogorov-Chentsov theorem, it can be shown that $B^H$ has H\"older continuous paths of any order less than $H$, see Section 1.4 of \cite{nourdin2012selected}. Let $(\sigma_s)_{s\in[0,1]}$ be $\alpha$-H\"older continuous with $\alpha>1-H$ and  $J=(J_t)_{t\in[0,1]}$ a  c\`adl\`ag jump process. 
In this case, Young proved that the pathwise Riemann-Stieltjes integral 
\begin{eqnarray*}
Y_t=\int_0^t\!\sigma_s\,dB_s^H, \quad  t\in[0,1],
\end{eqnarray*}
exists and is well-defined. Although one motivation is that $(X_t)$ provides a model for volatility processes, we shall refer to $(\sigma_t)$ as the volatility of our continuous process. In the application, it could rather model the volatility of the volatility then. We assume that this integral fractional processes $Y=(Y_t)_{t\in[0,1]}$ is observed with an additive jump component at equidistant discrete time points,
\begin{eqnarray*}
X_{k/n}=Y_{k/n}+J_{k/n}, \quad k=0,\ldots,n,
\end{eqnarray*}
over the unit time interval $[0,1]$. We develop asymptotic theory under infill asymptotics, such that the distance between neighboured observation times over the fix time interval tends to zero, $1/n\to 0$, as $n\to\infty$. As usually in a high-frequency framework, adding a bounded drift term to $(Y_t)$ would not affect our asymptotic results, such that our results could be extended, for instance, to the fractional Ornstein-Uhlenbeck model of \cite{Wang}. This is the case, since as $n\to \infty$, the increments 
\begin{align*}
\Delta_{n,k}^{(1)} X=(X_{k/n}-X_{(k-1)/n}), \quad k=1,\ldots,n,\end{align*}
and second-order increments
\begin{align*} \Delta_{n,k}^{(2)} X=(X_{k/n}-2X_{(k-1)/n}+X_{(k-2)/n}),\quad k=2,\ldots,n,\end{align*}
are dominated by jumps and the process $(Y_t)$, while increments of a drift term would be of order $n^{-1}$. We leave out a possible drift term for the sake of a simpler exposition. In the same way as for $(X_t)$, the notation for (second-order) increments is used for other processes. 

\subsection{Testing problem}
\begin{figure}[t]
\begin{center}
\includegraphics[width=6cm]{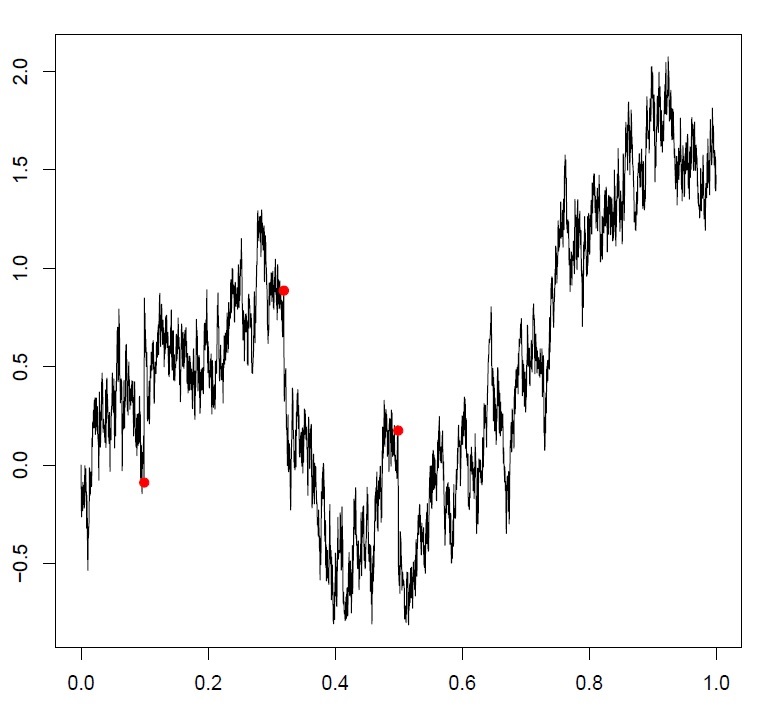}~\includegraphics[width=6cm]{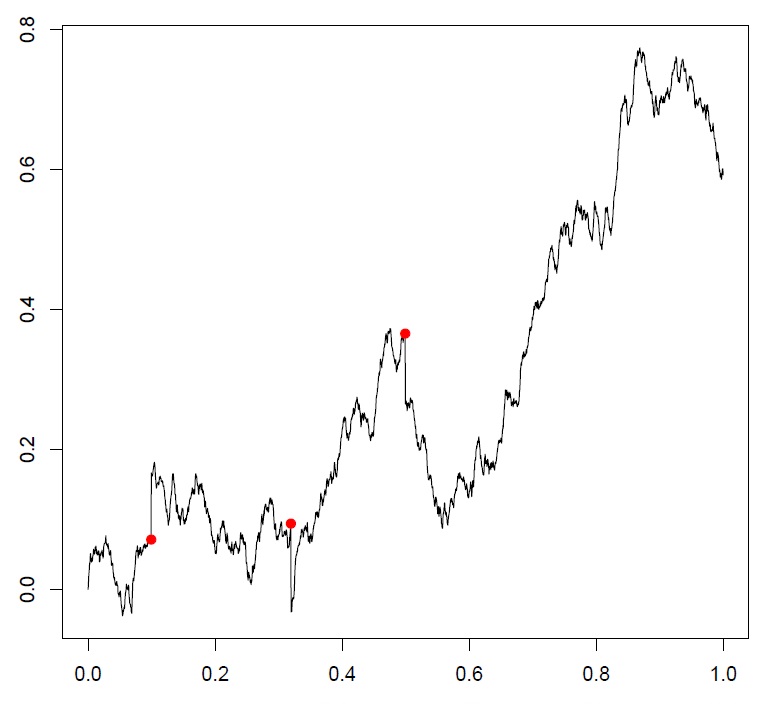}\\
\includegraphics[width=6cm]{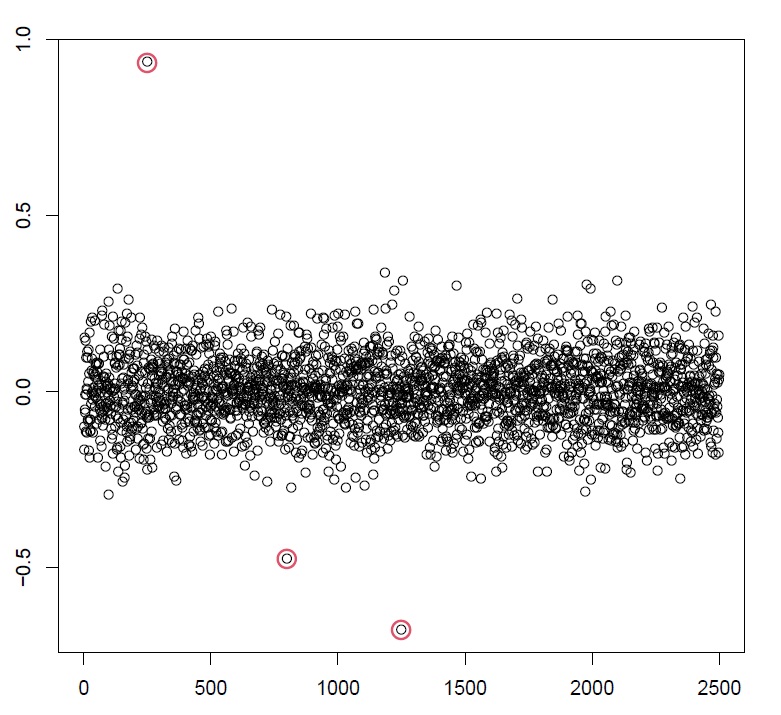}~\includegraphics[width=6cm]{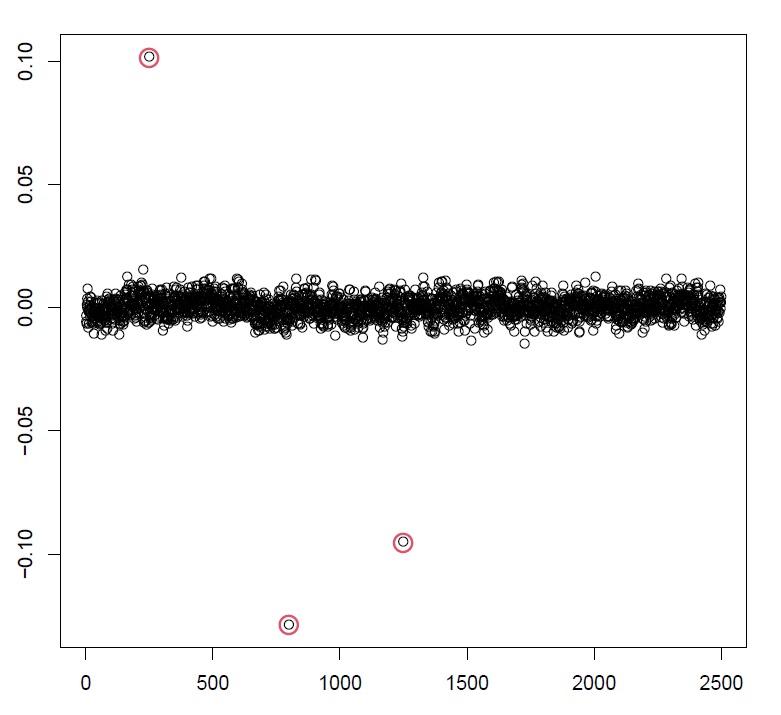}
\caption{\label{Fig:paths}Sample paths (top) and increments (bottom) with $H=0{.3}$ (left) and $H=0{.7}$ (right), and three jumps added which are highlighted by red circles and points.}
\end{center}
\end{figure}
We formalise our testing problem. When testing for positive jumps, we consider the null hypothesis $ H_0: \omega\in\Omega_0 $, and the alternative hypothesis $H_1:\omega\in\Omega_+$, where
\begin{align*}
\Omega_0=\left\lbrace \omega\in\Omega \mid J_t(\omega)= 0 \text{ for } t\in(0,1] \right\rbrace \quad \text{ and } \quad \Omega_+=\left\lbrace \omega\in\Omega\mid \exists t\in(0,1]: \Delta J_t(\omega)>0 \right\rbrace.
\end{align*}
In the same way, we can define a testing problem for negative jumps, which we omit, since we can treat negative jumps analogously. The testing problem for general jumps is defined with the null hypothesis  $ H_0: \omega\in\Omega_0 $ tested against the alternative hypothesis $H_1:\omega\in\Omega_1$, where
\begin{align*}
\Omega_0=\left\lbrace \omega\in\Omega \mid J_t(\omega)=0 \text{ for } t\in(0,1] \right\rbrace \quad \text{ and } \quad \Omega_1=\left\lbrace \omega\in\Omega\mid \exists t\in(0,1]:  \Delta J_t(\omega)\neq0 \right\rbrace.\end{align*}
It is not meaningful to consider a jump at time $0$ when we suppose c\`adl\`ag jumps. Figure \ref{Fig:paths} illustrates the intuitive fact that detecting jumps is more difficult when $H$ is small, when the sample paths of $B^H$ and $(Y_t)$ are rougher. Here, we add three jumps to two simulated paths of $B^H$, one with  $H=0{.3}$ and one with $H=0{.7}$. Comparing increments which include the jumps to other increments of $(Y_t)$ shows a much larger distance in the smoother example. The values on the $y$-axis also show the sizes of the three jumps which were set much smaller in the smoother case. 

\subsection{Regularity assumptions}
We work under the following regularity assumptions.
\begin{assumpsec}\label{Annahmen}
There is a constant $\alpha$, $0<\alpha\leq 1$, with $\alpha>1-H$, and a constant $K$, such that 
        \begin{eqnarray*}
        \vert \sigma_t-\sigma_s \vert\leq K\vert t-s \vert^{\alpha},\quad 0\leq s,t\leq 1,
        \end{eqnarray*}
				and $\sigma_s>0$, for all ${s\in[0,1]}$.
\end{assumpsec}
\clearpage 

One can analogously impose the conditions on the squared volatility $(\sigma_t^2)$. The current literature using rough fractional volatility for forecasting typically relies on a constant $\sigma_t=\sigma$, $t\in[0,1]$, see, e.g.\ \cite{forecast}. It is a main task for our methods and the asymptotic theory to include time-dependent $(\sigma_t)$. Due to the dependence structure of fractional processes and less available theory on these models compared to semimartingales, the proofs are different and one cannot expect results under similar minimal conditions as for semimartingales $(X_t)$, e.g.\ in \cite{palmes}. Note that the regularity of $(\sigma_t)$ is required to ensure that $(Y_t)$ is well-defined and hence analogous regularity is imposed in the literature on statistics for fractional processes, e.g.\ in \cite{corcuera2006power}. Our proofs of the main results Theorem \ref{Gumbelkonvergenz} and Theorem \ref{Divergenz der Teststatistik} do, however, not rely on the specific regularity $\alpha$ in any way, but only use the continuity.
\begin{assumpsec}\label{AnnahmeSpruenge}
 $(J_t)_{t\in[0,1]}$ is a jump process with c\`adl\`ag paths that satisfies for some $p\in(0,1]$ the condition
\begin{eqnarray*}
\mathbb{P}\left( \lim_{n\to\infty}\sum_{k=1}^n\vert \Delta_{n,k}^{(1)}J \vert^p<\infty \right)=1.
\end{eqnarray*}
\end{assumpsec}
The previous assumption is stronger for smaller $p\in(0,1]$. Jumps of finite variation as a minimal condition will be required to obtain a jump-robust spot volatility estimator for all $H\in(0,1)$. 

Throughout the manuscript, we use the notation $a_n\propto b_n$, for sequences $(a_n)_{n\in\N}$ and $(b_n)_{n\in\N}$, if $a_n/b_n$ tends to some positive constant and the standard notation $\mathcal{O}$ and $\KLEINO$ for Landau symbols, as well as $\mathcal{O}_{\mathbb{P}}$ and $\KLEINO_{\mathbb{P}}$ for the stochastic Landau symbols with respect to the probability measure $\P$.
\section{Construction of the tests\label{sec:3}}
In this section, we introduce the statistics our tests will be based on. Considering the maximal, standardized, absolute (second-order) increment the question of how to construct a test boils down to the question how to standardize before taking the maximum. The standardization is up to a scaling factor  $n^{H}$ basically a spot volatility estimate of $\sigma_t$. The standardization with an estimated volatility is crucial for several reasons. First, although the volatility is lower and upper bounded under Assumption \ref{Annahmen}, normalising with the local volatility is important to better detect jumps, since $n^{H}\Delta_{n,k}^{(1)} Y$ is approximately $\mathcal{N}(0,\sigma_{k/n}^2)$ normally distributed and time-varying volatility levels should not be neglected when comparing the size of the $k$th absolute increment to $n^{-H}\sigma_{k/n}$. Moreover, the standardization yields a pivotal test, that is, the limit distribution under the null hypothesis will not depend on the unknown volatility $(\sigma_t)$ any more. While in the semimartingale case we could simply consider a non-standardized version $n^{1/2}\max_k|\Delta_{n,k}^{(1)} Y|$, if we do not know $H$ the scaling factor $n^{H}$ is unknown and we exploit that the spot volatility estimates contain the same scaling factor $n^{H}$, which hence cancels out. Equivalently, this can be seen as a comparison of the size of an increment to many others in a neighbourhood.

While some aspects of our tests will be similar to the Gumbel tests within the semimartingale framework by \cite{leemykland}, the standardization is more crucial and more involved here. In particular, the spot volatility estimation needs to be robust with respect to jumps and we cannot use truncation methods to achieve this. We use power variations with small powers $p\le 1$, instead of the standard choice $p=2$, and second-order increments instead of increments for the jump-robust estimation of the spot volatility for all possible Hurst exponents $H\in(0,1)$. Denote $C_p$ the $p$-th absolute moment of a standard normal distribution, $\mathcal{N}(0,1)$, for some $p\in(0,1]$. For the window sizes of the local spot volatility estimation, let $(h_n)_{n\in\mathbb{N}}$ be a sequence of natural numbers with $\lim_{n\to\infty}h_n=\infty$ and $\lim_{n\to\infty}n^{-1}h_n=0$. In our notation, $n^{-1}h_n$ corresponds to the bandwidth of the nonparametric estimation. The test for positive jumps will be based on the test statistic
\begin{align}\label{Tn}
T_n=\max_{k,j}\left(\frac{\Delta_{n,j}^{(2)}X}{\left(h_n^{-1}C_p^{-1}\sum_{i=(k-1)h_n+2}^{kh_n} \Big\vert \Delta_{n,i}^{(2)}X \Big\vert^p\right)^{\frac{1}{p}}}\right).
\end{align}
To simplify the notation, we use the short notation
\begin{align*}
\max_{k,j}z_{k,j}:=\max_{2\leq j< h_n}z_{1,j}\vee \max_{2\leq k\leq m}\left(\max_{(k-1)h_n\leq j< kh_n}z_{k,j}\right)\vee \max_{mh_n\leq j\leq n}z_{m+1,j} \,,
\end{align*}
where $a\vee b=\max\{a,b\}$, and 
\begin{eqnarray*}
    \max_{k}z_k:=\max_{1\leq k\leq m}z_k,\quad  m=\lfloor n/h_n\rfloor\,,
\end{eqnarray*}
for some measurable transformations $z_{k,j}$ and $z_{k}$ of increments of our processes. Analogously, testing for jumps will be based on the statistic
    \begin{align}\label{Rn}
R_n=\max_{k,j}\left(\frac{\vert \Delta_{n,j}^{(2)}X \vert}{\left(h_n^{-1}C_p^{-1}\sum_{i=(k-1)h_n+2}^{kh_n} \Big\vert \Delta_{n,i}^{(2)}X \Big\vert^p\right)^{\frac{1}{p}}}\right).
\end{align}
Second-order increments are important in the denominators of \eqref{Tn} and \eqref{Rn} for the spot volatility estimation to ensure its good properties for large values of $H$. Then, however, we need to take second-order increments in the numerators as well, since $n^{H}\Delta_{n,k}^{(2)} Y$ is approximately $\mathcal{N}(0,\sigma_{k/n}^2 C_H)$ distributed, with a constant $C_H$ that hinges on the unknown $H$. Taking second-order increments in numerators and denominators of the ratios, the factors $C_H\cdot n^H$ cancel out. To ensure robustness against jumps which satisfy Assumption \ref{AnnahmeSpruenge}, we impose for our results the condition that $\lim_{n\to\infty}n^{pH}h_n^{-1}=0$. Since the statistician chooses $p$ and $h_n$, while $H\in(0,1)$ is an unknown parameter, smaller powers $p$ are preferable to allow for a choice of $h_n$ in favour of spot volatility estimates with small variance. An optimal window size $h_n$ in the normalization factor is given by $h_n\propto n^{2\alpha/(2\alpha+1)}$, with the regularity $\alpha$ of $(\sigma_t)$ from Assumption \ref{Annahmen}. This is obtained from the standard decomposition of the mean squared estimation error in variance and squared bias and balancing both parts. However, the asymptotic results for the tests will not require a rate-optimal volatility estimation and hence we do not need to assume that $\alpha$ is known. We will establish the main results under the condition that $h_n \propto n^{\beta}$, with some $\beta$, such that $pH < \beta<1$. This can be ensured by choosing $p$ sufficiently small. In any case, we have to use powers $p\le 1$, and cannot use the standard value $p=2$ as for realized volatilities.

The reason that power variations of second-order increments allow to estimate the volatility robust with respect to jumps without any truncation relates to a question which is currently of great interest for the literature on rough volatility models. The question is, if $H$ can be estimated robustly as well, and if ignoring jumps could manipulate estimates of $H$ when applying standard estimators with increments of $(X_t)$ inserted, while the estimator is actually built for observations of the continuous process $(Y_t)$. Given the importance of this aspect, we emphasize it in the next section before finishing the construction of our tests. 

\section{Jump-robust inference on rough processes\label{sec:4}}
For the construction of our test, we require a jump-robust spot volatility estimation. We use an estimator based on power variations of second-order increments and we do not use truncation or bipower variation statistics. The reason why this works relates to jump-robust inference on the Hurst exponent for rough processes based on high-frequency observations. Since we expect that this is relevant for the current research on rough volatility models, we emphasize here some crucial aspects. In particular, we consider in this section the standard discrete quadratic variation and a standard estimator for the Hurst exponent to analyse the effect of jumps on these statistics.

\begin{prop}\label{proprobust1}If $H\in(0,1/2)$, for any jump process $(J_t)$ with finite quadratic variation which is independent of $(Y_t)$, it holds under Assumption \ref{Annahmen} that
\begin{align}\label{intvola}n^{2H-1}\sum_{j=1}^n\big(\Delta_{n,j}^{(1)}X\big)^2=\int_0^1\sigma_s^2\,\intdiff s+\mathcal{O}_{\P}\big(n^{\max((2H-1), -1/2)}\big)\,.\end{align}
In particular, this implies that \(n^{2H-1}\sum_{j=1}^n\big(\Delta_{n,j}^{(1)}X\big)^2\stackrel{\P}{\rightarrow}\int_0^1\sigma_s^2\,\intdiff s\), as $n\to\infty$.
\end{prop}

This result is based on central limit theorems for power variations of integral fractional processes from \cite{corcuera2006power} and some estimates for the jumps. In fact, the possible robustness with respect to jumps was already mentioned by \cite{corcuera2006power}. Contrary to the situation for $H=1/2$, the mean of rescaled squared increments consistently estimates the integrated squared volatility, also in the presence of jumps. This works without truncation. In particular, if $H<1/4$, the central limit theorem for the integrated squared volatility with the standard optimal rate $n^{1/2}$ carries over. The limit theorem is then completely analogous to the continuous case when $(Y_t)$ is observed and quite general jumps in $(X_t)$ are asymptotically negligible. In particular, semimartingale jumps satisfy the assumptions. If $H>1/2$, the statistic left-hand side in \eqref{intvola} would diverge in the presence of jumps instead. For this reason, we use smaller powers $p$ for our statistical methods in \eqref{Tn} and \eqref{Rn}, which should work for all $H\in(0,1)$. While smaller $H$ and rougher paths of $(Y_t)$ make the detection of jumps more difficult, the effect of smaller $H$ on the robustness is positive instead. These effects can be seen as two sides of the same coin. In Figure \ref{Fig:paths} we see that the influence of increments with jumps becomes smaller compared to other increments when $H$ is smaller. It is thus natural that robustness of non-adjusted statistics -- when we ignore jumps -- is more likely to hold for smaller $H$.

Next, we point out that a standard estimator of the Hurst exponent is robust with respect to jumps without adjustments in the rough case when $H$ is small.
\begin{prop}\label{proprobust2}The estimator of the Hurst exponent,
\begin{align}\label{Hest}\hat H_n=\frac{1}{2\log(2)}\log\bigg(\frac{\sum_{j=0}^{n-2}\big(X_{(j+2)/n}-X_{j/n}\big)^2}{\sum_{j=0}^{n-1 }\big(X_{(j+1)/n}-X_{j/n}\big)^2}\bigg)\,,\end{align}
satisfies under Assumption \ref{Annahmen} for any $H\in(0,1/2)$, and for a jump process $(J_t)$, which is an It\^{o} semimartingale with finite quadratic variation, bounded jump sizes and independent of $(Y_t)$, that
\begin{align}\label{hatH}\hat H_n-H=\mathcal{O}_{\P}\big(n^{\max((2H-1), -1/2)}\big)\,.\end{align}
In particular, $\hat H_n$ is a consistent estimator for $H$, $\hat H_n\stackrel{\P}{\rightarrow} H$, as $n\to\infty$.
In case that $H\ge 1/2$, and if $(J_t)$ is an It\^{o} semimartingale with bounded jump sizes, it holds that $\hat H_n\stackrel{\P}{\rightarrow}1/2$, as $n\to\infty$.
\end{prop}
We use the notion of It\^{o} semimartingales in the sense of Section 2.1.4 of \cite{jacodprotter}. These are semimartingales whose characteristics are absolutely continuous with respect to the Lebesgue measure, such that the process admit a representation as Grigelionis processes in the sense of \cite{kallsen}.

Estimator \eqref{Hest} is a rather obvious estimator for $H$ based on power variations, and contained in the class of filtering estimators by \cite{coeurjolly2001estimating} setting $k=2$ and $a=(1,-1)$ in his general statistics. It is hence often used. With the estimate of $H$, a plug-in approach allows estimating the squared volatility. For $H<1/4$, the jump-robust estimator of the Hurst exponent attains the optimal rate of convergence, $n^{-1/2}$, and central limit theorems proved for the continuous case apply. Statistic \eqref{Hest} is indeed a proper estimator, also for unknown $H$ and $(\sigma_t)$, since $\hat H_n$ is a function of the observations only. If $(X_t)$ is used to model the volatility, we conclude that inference on the Hurst exponent under rough volatility is robust with respect to general additive jumps. This is an important insight for the current research on rough volatility. The second result in Proposition \ref{proprobust2}, that $\hat H_n\stackrel{\P}{\rightarrow}1/2$, also in the non-robust case of $H> 1/2$, is crucial to conclude that a small estimate of $H$ cannot be produced by jumps and a smoother continuous component. We conjecture that the robustness with respect to jumps can be extended even to a larger range of values of $H$, when using power variations with smaller powers and second-order increments. Since such results require more restrictive assumptions on the jumps and refined proofs, where less existing results can be exploited, we leave this conjecture open for future research. Here we focus on the impact of jumps on the standard methods which are typically used when jumps are ignored. 
Nevertheless, the robustness is only valid for small $H$ and our simulations and a data example show that jumps influence the finite-sample estimation. This provides additional motivation to construct jump-detection methods to filter out jumps in case that they are considered to be a nuisance quantity. The methods presented in the upcoming section allow testing for jumps and moreover to locate them and thus also to filter out jumps. Being aware of possible jumps, filtering increments with jumps based on our new methods and applying the standard estimators afterwards hence provides a tractable approach. In particular, for larger $H$, when the non-adjusted estimators for the Hurst exponent and the volatility do not work well, our methods to detect and filter jumps attain a particularly good performance. In this sense, the two reverse effects can be combined to accurately solve the problem of statistical inference across all model specifications.

\section{Asymptotic properties of the tests\label{sec:5}}
\subsection{Asymptotic distribution under the null hypothesis\label{sec:5.1}}
In this section, we present our original statistical methods and state the main results of this paper. For the tests, we establish the asymptotic behaviour of the statistics $T_n$ from \eqref{Tn} and $R_n$ from \eqref{Rn}.
Our first main result clarifies the asymptotic distributions of $T_n$ and $R_n$ under the null hypothesis that there are no jumps.
\begin{theo}\label{Gumbelkonvergenz}
Assume that $(\sigma_t)_{t\in[0,1]}$ satisfies Assumption  \ref{Annahmen}. In \eqref{Tn} and \eqref{Rn}, let $p\in(0,1]$ and set $h_n\propto n^{\beta}$, for $0<\beta<1$.
\begin{enumerate}
    \item[(i)] Under $H_0$, it holds that
\begin{align*}
\lim_{n\to\infty}\mathbb{P}\left( a_n\left(T_n-b_n \right)\leq x \right)=\exp(-\exp(-x)),\quad \text{ for } x\in\mathbb{R},
\end{align*}
with the sequences 
  \begin{align*}
a_n=\sqrt{2\log(n)},~\text{and}~~ b_n=\sqrt{2\log(n)}-\frac{\log(\log( n))+\log(4\pi)}{2\sqrt{2\log(n)}}\,.
\end{align*}

    \item[(ii)] Under $H_0$, it holds that
\begin{align*}
\lim_{n\to\infty}\mathbb{P}\left( c_n\left(R_n-d_n \right)\leq x \right)=\exp(-\exp(-x)),\quad\text{ for } x\in\mathbb{R},
\end{align*}
with the sequences 
  \begin{align*}
	c_n=\sqrt{2\log(2n)},~\text{and}~~ d_n=\sqrt{2\log(2n)}-\frac{\log(\log(2n))+\log(4\pi)}{2\sqrt{2\log(2n)}}.
\end{align*}
\end{enumerate}
\end{theo}
The pointwise convergence of the cumulative distribution function shows that convergence in distribution to a standard Gumbel limit distribution is satisfied. The sequences $(a_n)_{n\in\N}$ and $(b_n)_{n\in\N}$ are identical to the ones in the Gumbel convergence of the maximum of i.i.d.\ standard normal random variables and thus also agree to the ones occurring in \cite{leemykland}. For the sequences $(c_n)_{n\in\N}$ and $(d_n)_{n\in\N}$, factors $n$ are replaced by $2n$, similar as in the Gumbel convergence of the maximum of absolute values of i.i.d.\ standard normal random variables. Let $q_{\alpha}$ be the $(1-\alpha)$-quantile of the Gumbel distribution:
\begin{align*}
q_{\alpha}=-\log\left(-\log\left(1-\alpha\right)\right).
\end{align*}
\begin{enumerate}
\setlength{\itemsep}{.5em}
\item[(T1)] Testing for positive jumps, we reject $H_0$, if $a_n(T_n-b_n)\geq q_{\alpha}$.
\item[(T2)] Testing for jumps, we reject $H_0$, if  $c_n(R_n-d_n)\geq q_{\alpha}$.
\end{enumerate}
For these tests, Theorem \ref{Gumbelkonvergenz} readily yields the following asymptotic properties.
\begin{cor}\label{asymptotic level}
\begin{enumerate}
\item[(i)] The test $\operatorname{(T1)}$ for $H_0$ against $H_1:\omega\in\Omega_+$ has asymptotic level $\alpha$, as $n\to\infty$.
\item[(ii)] The test $\operatorname{(T2)}$ for $ H_0$ against $H_1:\omega\in\Omega_1$ has asymptotic level $\alpha$, as $n\to\infty$.
\end{enumerate}
\end{cor}
\subsection{Consistency and rate of convergence under the alternative\label{sec:5.2}}
For the test, we next clarify the behaviour of the statistics under the alternative hypothesis.
\begin{theo}\label{Divergenz der Teststatistik}

    Assume that $(\sigma_t)_{t\in[0,1]}$ satisfies Assumption \ref{Annahmen} and that Assumption \ref{AnnahmeSpruenge} is satisfied for some $p \in (0,1]$. Choose $\beta\in (pH,1)$ and $h_n\propto n^{\beta}$.

\begin{enumerate}
    \item[(i)] Under the alternative hypothesis $H_1:\omega\in\Omega_+$, it holds true that
    \begin{eqnarray*}
    \lim_{n\to\infty}\mathbb{P}\left( n^{-\gamma}a_n\left(T_n-b_n \right)\leq L \right)=0,\quad \text{ for all }\quad L\in\mathbb{R} \quad \text{ and }\quad 0<\gamma<H.
    \end{eqnarray*}
    \item[(ii)] Under the alternative hypothesis $H_1:\omega\in\Omega_1$, it holds true that
    \begin{eqnarray*}
    \lim_{n\to\infty}\mathbb{P}\left(n^{-\gamma} c_n\left(R_n-d_n \right)\leq L \right)=0,\quad\text{ for all }\quad L\in\mathbb{R}\quad \text{ and }\quad 0<\gamma<H.
    \end{eqnarray*}
\end{enumerate}
\end{theo}

Theorem \ref{Divergenz der Teststatistik} readily implies the following asymptotic properties of the Gumbel tests.
\begin{cor}\label{asymptotic power}
\begin{enumerate}
\item[(i)] The test $\operatorname{(T1)}$ for $H_0$ against $H_1:\omega\in\Omega_+$ is consistent, that is, it has asymptotic power $1$, as $n\to\infty$.
\item[(ii)] The test $\operatorname{(T2)}$ for $H_0$ against $H_1:\omega\in\Omega_1$ is consistent, that is, it has asymptotic power $1$, as $n\to\infty$.
\end{enumerate}
\end{cor}

Moreover, Theorem \ref{Divergenz der Teststatistik} establishes a rate of convergence. It is a standard concept for asymptotic tests in asymptotic statistics to formulate the convergence rate via local alternatives. This means for increasing sample size, we consider a sequence of alternatives with decreasing distance to the null hypothesis. Here, we conclude the consistency of the test in case of asymptotically decreasing sequences of (absolute) jump sizes $(|\Delta J_{\tau}|)_n$, as long as $\liminf_{n\to\infty} (|\Delta J_{\tau}|)_n n^{\gamma}>0$, with arbitrary $\gamma<H$. Since $n^{-H}$ is the size of absolute increments of the continuous component $(Y_t)$, a faster rate is not possible. As expected, the rate hinges on $H$ and is faster for larger $H$. In Figure \ref{Fig:paths}, we can thus find much smaller jumps in the smoother case on the right than in the rougher example left-hand side.

\subsection{Localization of jumps\label{sec:5.3}}
One crucial advantage of tests for jumps based on maximum statistics is that they readily allow the localization of jumps. That is, the associated $\operatorname{argmax}$ yields a consistent estimator for the time at which a jump occurred. We establish consistency and a fast rate of convergence of order $n$, under the alternative $H_1$, when one jump occurred.
\begin{prop}\label{proplocalization}
Under $H_1$, if there is one jump at time $\tau\in(0,1)$, and $\Delta J_t=0$, for all $t\ne \tau$, the estimator of the jump time
\[\hat\tau_n=\frac{1}{n}\operatorname{argmax}_{2\le j\le n }\frac{|\Delta_{n,j}^{(2)} X|}{\Big(h_n^{-1}C_p^{-1}\sum_{i=\lfloor jh_n^{-1}\rfloor}^{\lfloor jh_n^{-1}\rfloor+h_n}|\Delta_{n,j}^{(2)} X|^p\Big)^{1/p}}\,,\]
satisfies under the conditions of Theorem \ref{Divergenz der Teststatistik} that
\[\big(\hat\tau_n-\tau\big)=\mathcal{O}_{\P}\big(n^{-1}\big)\,.\]
\end{prop}
Clearly, based on discrete observation times with distance $n^{-1}$, it is impossible to locate jumps more accurately than in an interval of length $n^{-1}$ around $\tau$. As the proof of the proposition shows, this works even for decreasing sequences of absolute jump sizes $(|\Delta J_{\tau}|)_n$, as long as $\liminf_{n\to\infty} (|\Delta J_{\tau}|)_n n^{\gamma}>0$, for some $\gamma<H$. Thus, we can estimate jump times with the best possible rate of convergence. Confidence for $\operatorname{argmax}$-estimators is, however, a very involved problem which is beyond the scope of this manuscript.\\
A sequential application of the test based on the maximum and the $\operatorname{argmax}$-estimator, where we discard in each step the previous maximal absolute (second-order) increment, yields a sequential top-down algorithm to estimate the number and times of jumps under $H_1$, if we assume some finite number of jumps. Furthermore, the maximal absolute increments and their signs yield estimates of jump sizes as well.

\section{Simulations\label{sec:6}}
\begin{figure}[t]
\begin{center}
\includegraphics[width=12cm]{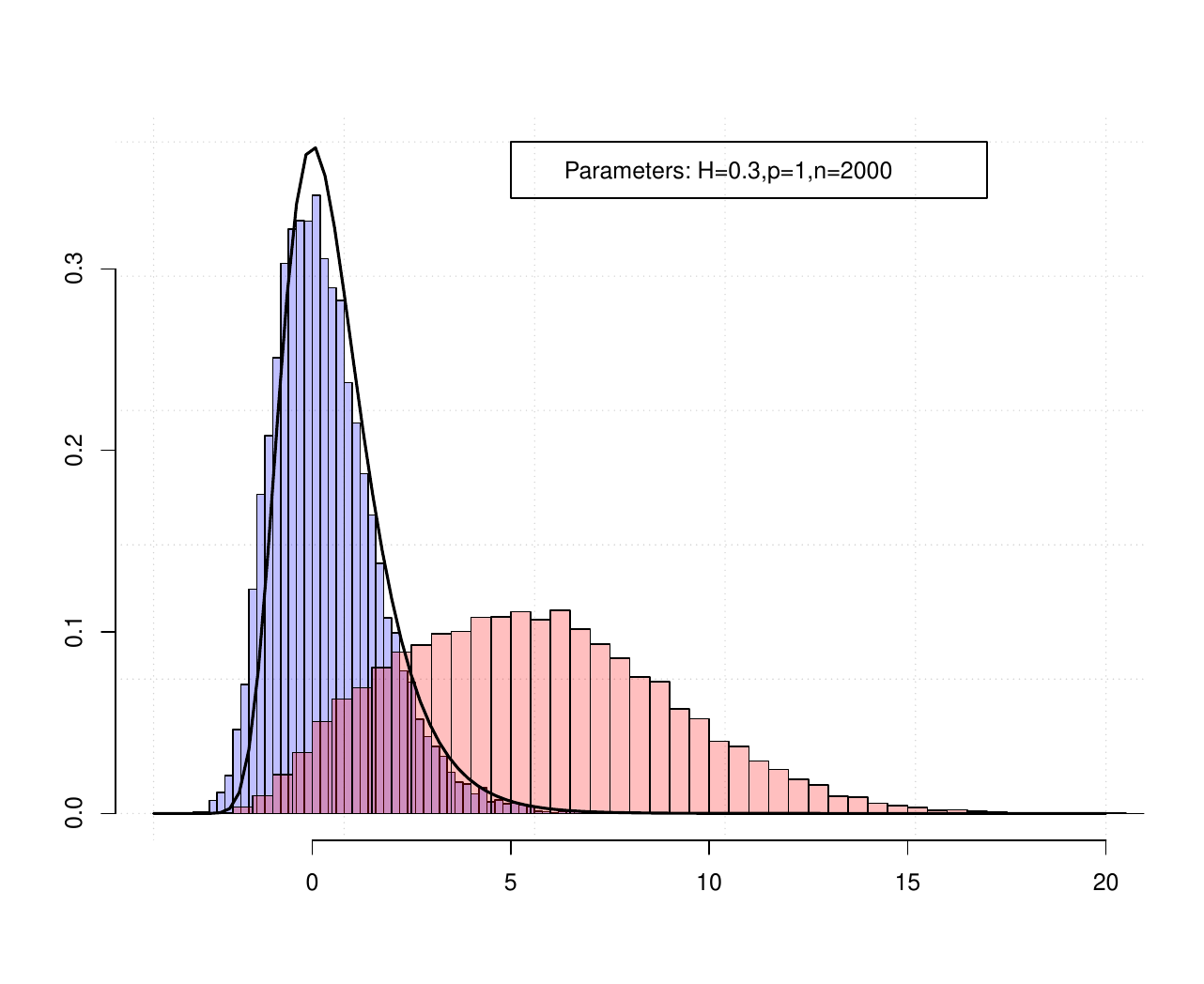}
\caption{Empirical distributions of test statistic under the null hypothesis, compared to the standard Gumbel limit distribution drawn with the black line, and under the alternative hypothesis with fix jump size $\Delta J_{\tau}=0{.}8$. \label{Fig:histo}}
\end{center}
\end{figure}

In this section, we investigate the finite-sample properties of our test and the estimation of the Hurst exponent in the presence of jumps in a Monte Carlo simulation study. We implement the model with the volatility function
\begin{eqnarray*}
\sigma_t=1-0{.}2\cdot \sin\left(\frac{3}{4}\pi t\right),\quad t\in[0,1],
\end{eqnarray*}
which is a Lipschitz continuous function, such that we have $\alpha=1$ in Assumption \ref{Annahmen}. Therefore, we choose the optimal window size $h_n=\lfloor n^{2/3}\rfloor$, with $\beta=2/3$, for spot volatility estimation. For the simulation of paths of the fractional Brownian motion, we are using the Cholesky method as described in Section 3.4 of \cite{coeurjolly2000simulation}. We will only illustrate the finite-sample properties of the test $\operatorname{(T1)}$ based on $T_n$ from \eqref{Tn}. The results for the test $\operatorname{(T2)}$ based on $R_n$ from \eqref{Rn} are completely analogous and hence omitted.

Figure \ref{Fig:histo} visualizes the empirical distribution of the test statistic for the test $\operatorname{(T1)}$ in the left histogram under the null hypothesis. The histograms are standardised to densities and we draw a comparison to the density of the theoretical, standard Gumbel limit distribution which is drawn with a solid black line. Both histograms in Figure \ref{Fig:histo} are based on 2000 Monte Carlo runs of our model with $H=0{.}3$, and for sample size $n=2000$. The empirical quantiles match reasonably well with those of the Gumbel density. Although the finite-sample fit is not perfect, in particular the large quantiles closely track their theoretical asymptotic counterparts. We conclude that we can use the test as constructed without finite-sample adjustments as, for instance, a bootstrapped version. The right histogram in Figure \ref{Fig:histo} shows the empirical distribution of the test statistic under the alternative hypothesis with a fixed jump size of $0{.}8$, at a generated jump time $\tau$, which is uniformly distributed on $\lbrace 1/n,\ldots,(n-1)/n \rbrace$. In this setting, the empirical distributions under the null and alternative hypotheses separate, but the power of the test does not attain a value close to 1. We can see this in Figure \ref{Fig:histo}, since the two histograms overlap.

\begin{figure}[t]
\begin{center}
\includegraphics[width=12cm]{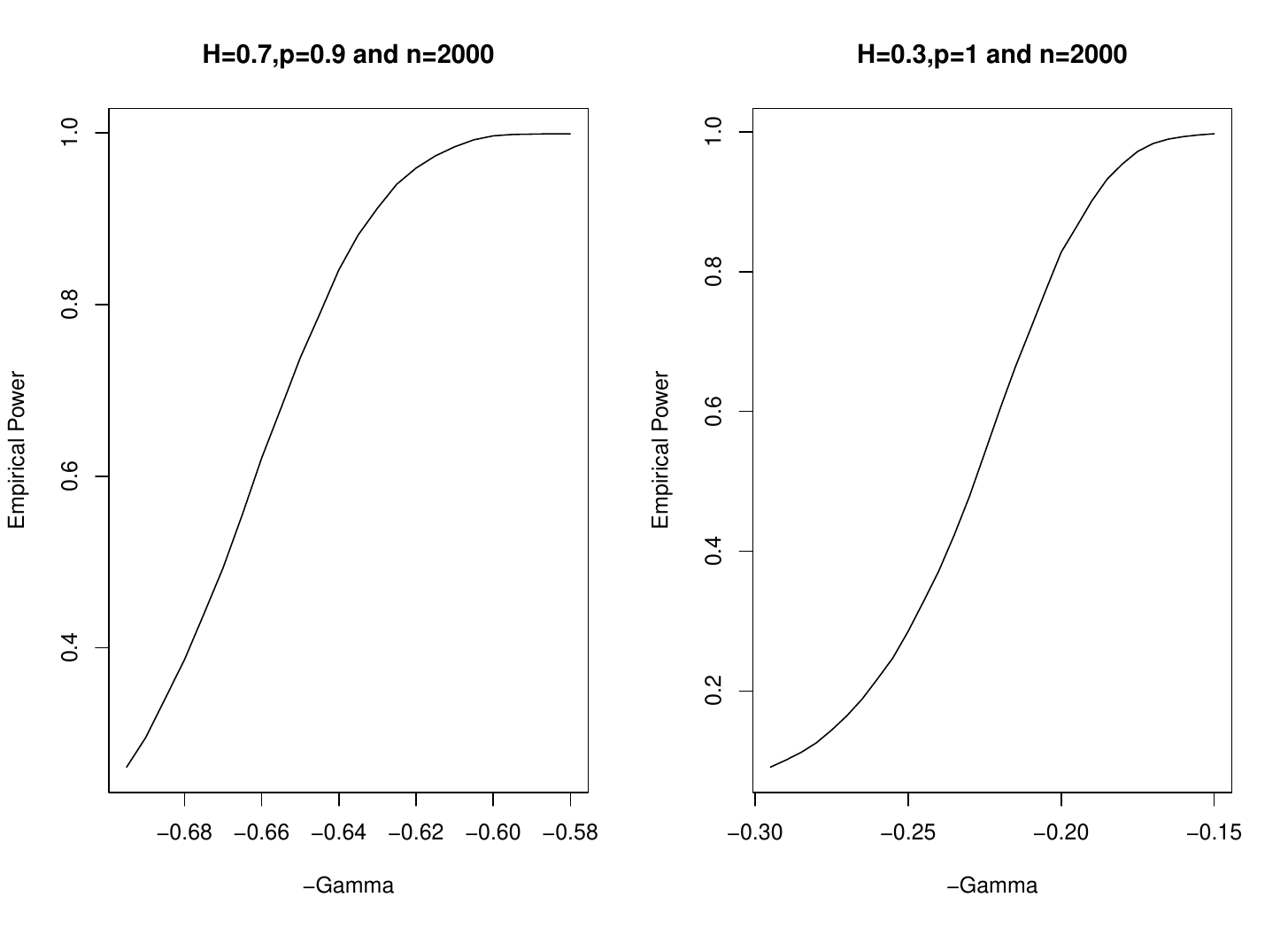}
\caption{Empirical power against level $\gamma$ of jump sizes $\sqrt{2\log(n)}n^{-\gamma}$. \label{Fig:power}}
\end{center}
\end{figure}

We consider one jump at a uniformly distributed jump time to demonstrate the power of our test in a finite sample. The empirical power of the test for positive jumps with level 5\% is illustrated for jump sizes $\sqrt{2\log(n)}n^{-\gamma}$, with different values of $0<\gamma<H$, in Figure \ref{Fig:power} with a sample size $n=2000$. One of our main theoretical results shows that for jump sizes with $\gamma<H$ our test $\operatorname{(T1)}$ should start to work with increasing power as $\gamma$ decreases. Figure \ref{Fig:power} confirms the expected behaviour, for $H=0{.}3$ right-hand side and $H=0{.}7$ left-hand side, and the power increases rapidly with decreasing values of $\gamma$. We use statistic \eqref{Tn} with $p=1$ for $H=0{.}3$, and with $p=0{.}9$ for $H=0{.}7$, such that the condition $pH<\beta=2/3$ from Theorem \ref{Divergenz der Teststatistik} is satisfied in both cases. Using smaller values of $p$ yields similar empirical results. Comparing the associated jump sizes for certain powers in the two plots of Figure \ref{Fig:power} underlines the fact that larger $H$, resulting in smoother paths of $(Y_t)$, allows to detect much smaller jumps. The setting of Figure \ref{Fig:histo} corresponds to $\gamma\approx 0{.}21$, and the plot right-hand side in Figure \ref{Fig:power} shows that the power is slightly below 80\% in this case. 

Tables \ref{Empirical power H=0.7} and \ref{Empirical power H=0.3} contain the empirical powers for different jump sizes and different sample sizes. In each scenario, the jump size is kept fix and the time of the jump is uniformly distributed. In each column, for fixed jump sizes, the empirical power increases from top down when the sample size is increasing. In each row, for fixed sample sizes, the empirical power increases from left to right when the jump size is increasing.

\begin{table}[ht]
    \centering
    \begin{tabular}{|l|c|c|c|c|c|c|c|c|c|c|}
  \hline
  \diagbox{$n$}{$\Delta J_{\tau}$}
                   &$0.013 $ & $ 0.015 $ & $ 0.017 $ & $ 0.019$ & $0.020 $ & $0.022 $ & $ 0.024$ & $ 0.026 $ & $0.028 $ &$0.029 $\\
  \hline
  $500$   &0.046  &0.046 & 0.046 & 0.046 &  0.047 & 0.049 & 0.051 & 0.053 & 0.058 & 0.063\\ \hline
  $1000$    &0.049  &0.052 & 0.054 & 0.059 & 0.065 & 0.091 & 0.101 & 0.127 & 0.163  &0.200\\ \hline
  $1500$    &0.049& 0.060& 0.081& 0.107& 0.143& 0.197& 0.255& 0.341& 0.434& 0.519\\ \hline
  $2000$ &0.073& 0.107& 0.148& 0.210&  0.287& 0.394& 0.504& 0.616& 0.730& 0.825\\ \hline
  $2500$    &0.102& 0.166& 0.251&  0.362& 0.493&  0.625& 0.751& 0.846& 0.921& 0.961\\ \hline
\hline

\end{tabular}
\caption{Empirical power for Hurst exponent $H=0{.}7$ based on $2000$ Monte Carlo simulations. $\Delta J_{\tau}$ denotes the jump size, starting around $2500^{-0.67}\sqrt{2\log(2500)}\approx 0.013$. We use \eqref{Tn} with $p=0{.}9$. \label{Empirical power H=0.7}}
\vspace*{.25cm}

  \begin{tabular}{|l|c|c|c|c|c|c|c|c|c|c|}
  \hline
  \diagbox{$n$}{$\Delta J_{\tau}$}
                   &$0.360 $ & $0.415  $ & $ 0.470 $ & $ 0.525 $ & $ 0.580 $ & $ 0.635 $ & $ 0.690$ & $ 0.745$  & $0.800$ &$  0.855 $\\
  \hline
  $500$      &0.047& 0.053& 0.059& 0.080& 0.098& 0.118& 0.152& 0.188& 0.237& 0.289\\\hline
  $1000$       & 0.057& 0.069& 0.083& 0.109& 0.156& 0.199& 0.279& 0.346& 0.449& 0.551\\\hline
   $1500$       & 0.059& 0.075& 0.102& 0.147& 0.205& 0.272& 0.373& 0.503& 0.613& 0.708\\ \hline
    $2000$       & 0.070& 0.096& 0.140& 0.209& 0.296& 0.400& 0.515& 0.628& 0.740& 0.826\\\hline
     $2500$       & 0.072& 0.105& 0.161& 0.244& 0.353& 0.477& 0.617& 0.733& 0.828& 0.910\\\hline
  \hline
\end{tabular}
    \caption{Empirical power for Hurst exponent $H=0{.}3$ based on $2000$ Monte Carlo simulations. $\Delta J_{\tau}$ denotes the jump size, starting around $2500^{-0.25}\sqrt{2\log(2500)}\approx 0{.}36$. We use \eqref{Tn} with $p=0{.}9$. \label{Empirical power H=0.3}}
\end{table}
To evaluate our test, we further compare the rejection rates under the null hypothesis for the test $\operatorname{(T1)}$ used with the quantile $q_{0.95}=-\log(-\log(0.95))$ of the theoretical Gumbel limit distribution. Table \ref{Empirical level} contains the empirical levels, that is, the size of the test, for different Hurst exponents and different sample sizes. These values indicate that our test attains the corresponding level in a finite sample. In fact, the empirical rejection rates are even slightly smaller than the 5\% level.
\begin{table}[ht]
    \centering
    \begin{tabular}{|l|c|c|c|c|c|}
  \hline
  \diagbox{$H$}{$n$}
                 & $500$ &$1000$ & $1500$ & $2000$ & $2500$ \\
  \hline
  $0.3$       &    0.0306    & 0.0418         &  0.0380          & 0.0392    &  0.0372          \\ \hline
  $0.7$       &    0.0430    & 0.0428         &  0.0408          & 0.0386    &  0.0348           \\ \hline
  \hline
\end{tabular}
    \caption{Empirical size of the test calculated from $5000$ Monte Carlo simulations. For $H=0{.3}$ we use $p=1$ and and for $H=0{.}7$ we use $p=0{.}9$ in \eqref{Tn}.}
    \label{Empirical level}
\end{table}

\begin{figure}[t]
\begin{center}
\includegraphics[width=14.5cm]{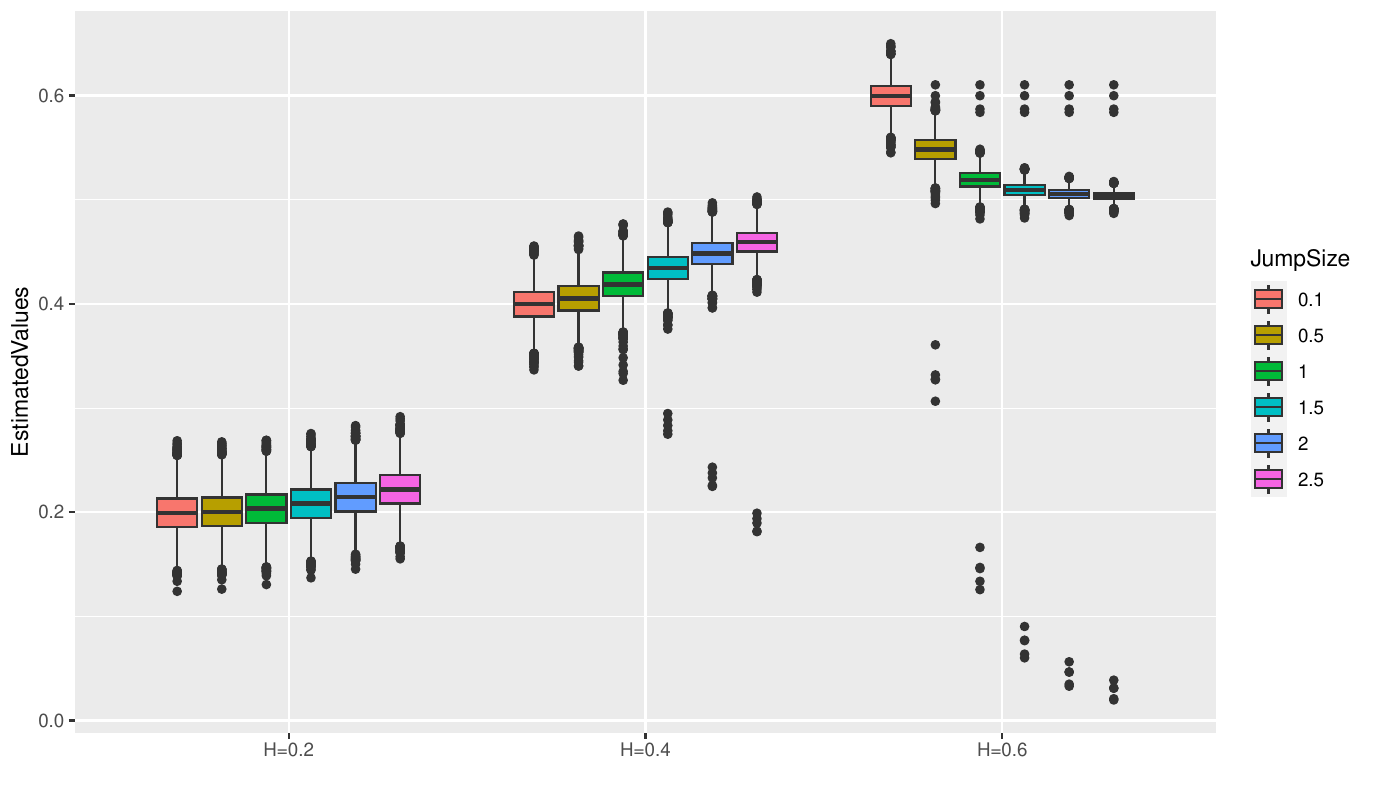}
\caption{Empirical robustness in estimating the Hurst exponent. \label{Fig:robust}}
\end{center}
\end{figure}

We study the finite-sample distribution of the estimator \eqref{Hest} to analyse its robustness with respect to jumps. Figure \ref{Fig:robust} shows boxplots of the estimates for three different values of the Hurst exponent, $H=0{.}2$, $H=0{.}4$, and $H=0{.}6$. For each value of $H$, we add one jump to the simulated paths of $(X_t)$ at a uniformly distributed jump time and we consider an increasing sequence of jump sizes starting with a very small one of size $\Delta J_{\tau}=0{.}1$, up to a huge jump of size $\Delta J_{\tau}=2{.}5$. Each boxplot is based on 2000 Monte Carlo runs and the sample size is in all scenarios $n=2000$. We directly apply estimator \eqref{Hest}, without filtering jumps, to demonstrate the effects of ignoring jumps on the estimation of the Hurst exponent. This seems to be important to us in view of the existing literature and empirical evidence for rough volatility.

For $H=0{.}2$, by the theoretical result of Proposition \ref{proprobust2} the estimator is consistent, also in presence of jumps, and even satisfies the central limit theorem at optimal rate $\sqrt{n}$. The robustness is confirmed in the finite-sample simulation left in Figure \ref{Fig:robust}, at least for moderate jump sizes. Huge jumps can, however, result in a positive finite-sample bias. One implication of this finding is that filtering jumps based on our methods is beneficial for estimating the Hurst exponent and can increase the finite-sample precision, also for small values of the true Hurst exponent. Another important insight is that a finite-sample bias due to a jump is positive and thus jumps should not manipulate estimates of $H$ in the way that smaller estimates are obtained. A positive bias is expected from the theory, since the jumps inserted in the estimator \eqref{Hest} yield the value 1/2. Therefore, the empirical means should lie between 1/2 and the unbiased estimates based on the observations of the continuous component $(Y_t)$. This is good news for the literature pointing at empirical evidence for rough volatility. Ignoring jumps should not result in a negative bias of their estimated Hurst exponents, at least when using the estimator \eqref{Hest} or similar ones.

For $H=0{.}4$, by Proposition \ref{proprobust2} consistency of the estimator still holds. The convergence rate is however slower in these scenarios and the central limit theorem which would hold for observations of $(Y_t)$ without jumps does not remain valid in presence of jumps. The few outliers we see for $H=0{.}4$ and $H=0{.}6$ in Figure \ref{Fig:robust} are due to jumps with a jump time that falls in the last observed increment. In this case, the jumps inserted in the estimator \eqref{Hest} yield the value 0 instead of 1/2, what explains the smaller estimates in these cases. Compared to the smaller Hurst exponent, the positive finite-sample bias with increasing jump size is much more pronounced here. Filtering jumps before the estimation of $H$ is thus even more important. The localization based on our methods from Section \ref{sec:5.3} yields the correct jump times, referring to the correct discrete index of the second-order increments containing the jumps, in basically all cases in that the test rejects correctly. Since the power of our test is high for the larger jump sizes, pre-filtering jumps with our methods would correct the finite-sample bias in the estimates of Figure \ref{Fig:robust}.

Finally, on the right in Figure \ref{Fig:robust} we apply estimator \eqref{Hest} to paths of $(X_t)$ with jumps and $H=0{.}6$. Proposition \ref{proprobust2} shows that the estimator is here inconsistent and converges to 1/2 in probability instead of $0{.}6$. We see this for larger jumps and that the variance of the estimation becomes small when the jumps dominate the continuous component in the finite-sample, empirical distribution. Only for very small jumps, the estimation is still robust with respect to jumps for this fixed sample size.

In conclusion, Figure \ref{Fig:robust} underlines our finding that ignoring jumps does not manipulate the empirical evidence for rough volatility obtained so far. At the same time, it motivates that filtering jumps is nevertheless important, also when estimation of $H$ is the main target. Our proposed test with the localization procedure can be used to filter out the jumps. After discarding the largest absolute second-order increments which our test ascribes to jumps only the remaining second-order increments should be used for estimators as \eqref{Hest}.

\section{Data application\label{sec:6neu}} 
\begin{figure}[t]
\begin{center}
\includegraphics[width=7.25cm]{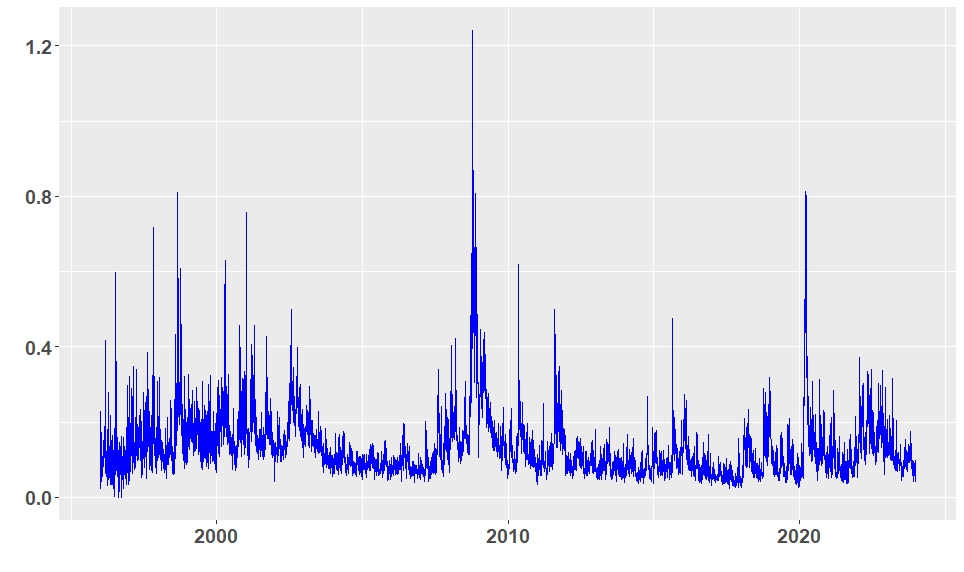}\includegraphics[width=7.25cm]{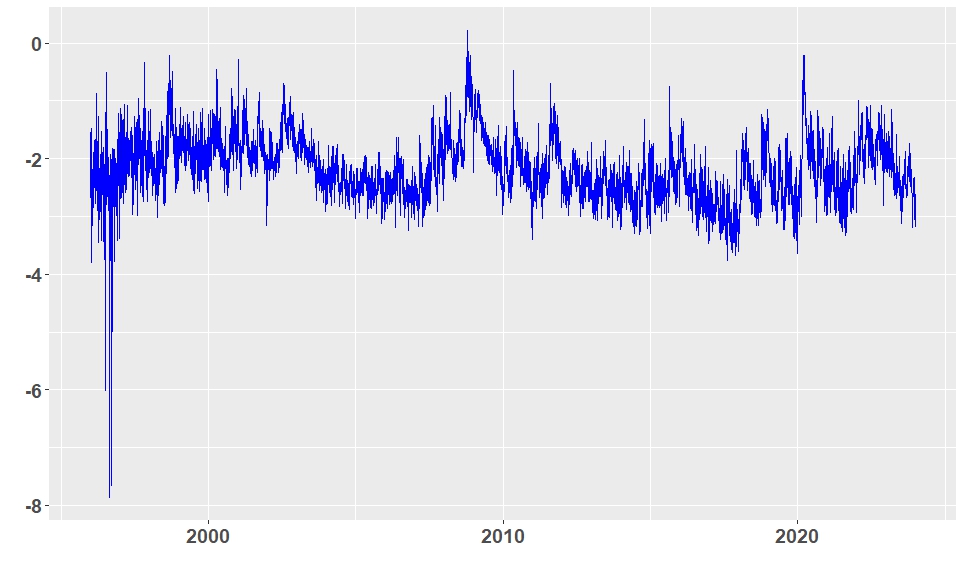}
\caption{Example time series of daily volatilities (left) and logarithmic daily volatilities (right).\label{Fig:dataexample}}
\end{center}
\end{figure}

For a relevant data example, we consider a time series of daily volatility estimates inferred from intra-day high-frequency data of the S\&P 500 market ETF (SPY). The data is constructed from the Risk Lab on Dacheng Xiu's website.\,\footnote{\href{https://dachxiu.chicagobooth.edu/\#risklab}{https://dachxiu.chicagobooth.edu/\#risklab}} Precisely, we use the daily quasi maximum likelihood estimates (QMLE) based on trade prices from 1996 to 2023. The time series includes in total 7021 observations. It covers as a subsample the time span from October 1, 2017, to September 30, 2022. This is considered in \cite{forecast} for an empirical study including estimation and forecasting based on a parametric model with a fractional Brownian motion with constant volatility. In line with the literature, we model the observed log-volatilities by a fractional process. Our model is more general than the one of \cite{forecast} in two aspects, allowing for time-varying volatility and taking possible jumps into account. Time-varying and constant volatility refer here to the underlying volatility of the observed log-volatility time series. Figure \ref{Fig:dataexample} illustrates the time series with the volatility estimates left-hand side and the logarithmic values right-hand side. In particular during the subprime mortgage crisis in the year 2008, the plots show a high volatility cluster. In this period, we might expect to find some jumps. Moreover, a few very small volatilities during 1996 result in low peaks of the log-volatilities with negative values of large absolute sizes. The turbulent dynamics of the overall time series suggest that a parametric model with constant volatility (of volatility) might be too restrictive. Our model allows to estimate a time-dependent volatility (of volatility), e.g.\ higher in periods of financial crises, to avoid spurious jump detections. It is less clear if jumps and time-varying volatility (of volatility) are important in the five year subsample from \cite{forecast}. We observe in Figure \ref{Fig:dataexample} that the beginning of the Corona pandemic triggered another cluster of higher volatility in spring 2020.
 
Our implementation of the methods used for this study is available,\,\footnote{\href{https://github.com/Michael-Sonntag/Gumbel-test}{at github.com/Michael-Sonntag/Gumbel-test}} such that all results are reproducible. An application of our test  $\operatorname{(T2)}$ to the time series yields rejection of the null hypothesis that there is no jump at any reasonable level, e.g.\ $\alpha=5\%$, or $\alpha=1\%$. Here and in the sequel, we use the statistic \eqref{Rn} with the power $p=0{.}85$ and $h_n=\lfloor n^{2/3}\rfloor$. The outcome of the test is, however, robust with respect to different choices. The null hypothesis is not rejected when we apply the test $\operatorname{(T2)}$ only to the subsample considered in \cite{forecast}. We conclude that in general jumps should be taken into account and filtered out before fitting a continuous, rough fractional volatility model to this data. At the same time, the non-rejection for the subsample shows that our method is not over-sensitive and confirms that the results by \cite{forecast} remain valid when taking possible jumps into account.

In order to locate all jumps in the time series, we perform a sequential application of the test $\operatorname{(T2)}$. That is, after rejection, we estimate the time of the largest jump with the $\operatorname{argmax}$-estimator from Proposition \ref{proplocalization}. In the next step, we repeat the test discarding the largest second-order increment that was already detected to be due to a jump. If the test rejects once more, we locate the next jump based on the $\operatorname{argmax}$-estimator. We repeat the procedure until the test does not reject the null hypothesis any more. Based on this sequential application of the test at level 5\% (95\% confidence), we find in total 14 second-order increments which are ascribed to jumps (7 at level 1\% and 15 at level 10\%). From the 14 detected second-order increments with jumps, as expected, three are in late September and October 2008. Even 8 of them are related to the small volatilities in 1996. However, note that the number of detected jumps is in fact not exactly the same as the number of affected second-order increments, but should be smaller than 14. The reason is that a jump enters two neighbouring second-order increments such that typically both are affected by the same jump. We find some stylized facts which further complicate the empirical analysis. We detect some log-volatility outliers, associated with the small volatilities in 1996. In contrast to the stylized picture of a directional jump, an outlier is a more isolated large absolute value neighboured by much smaller ones before and afterwards. Hence, an outlier is associated rather with two successive jumps in opposite directions. We emphasize this stylized fact as a caveat, since under high-frequency asymptotics directional jumps are expected whose inter-arrival times include several continuous movements rather than outliers. In view of the daily time steps, however, this stylized fact might not be too surprising. Moreover, it does not cause problems when using our methods for jump filtering. We spot one outlier on September 20, 1996, which affects three neighboured second-order increments, while being associated with two jumps. On August 19 and 20, 1996, there are two outliers next to each other affecting four second-order increments, including three of our 14 detected ones and one which is slightly below the critical value of the test. Taking these specific events into account, we count in total 11 estimated jumps at level 5\%. The picture of a directional jump which results in two detected neighboured second-order increments occurred in fact only once. In other cases, neighbours were below the critical values and hence not included. In Figure \ref{Fig:dataexample} we plot the standardized second-order increments over which the maximum is taken in statistic \eqref{Rn}. The second-order increments with detected jumps by our sequential application of the test are highlighted in blue. We grey out second-order increments which are only affected as neighbours of highlighted ones and associated with the same jumps. This includes the largest second-order increment which adds the sizes of two neighboured jumps around the outlier. This detail does not influence the outcome of the test.

We find that it is important to consider time-varying volatility (of volatility) for the jump detection. Assuming it constant instead and standardizing thus with one global estimate would result in a larger number of estimated jumps, i.e., 35 second-order increments at level 5\% (24 at level 1\% and 41 at level 10\%). For the subsample considered in \cite{forecast}, the test does not reject the null hypothesis, also not when we assume constant volatility. 

Let us emphasize an important detail of the sequential application of the test. In fact, the Gumbel test by \cite{leemykland} has been used in the way of a thresholding method. That is, a threshold is computed based on the quantile of the Gumbel limit law and the normalizing sequences and each increment exceeding this threshold is detected as a jump. This simple ad-hoc procedure can be quite similar to a sequential application of the Gumbel test, but it is actually not equivalent. For this difference it is not important if using first- or second-order increments. The reason is that after each test, we compute the statistic \eqref{Rn} and the critical value for the test based on a reduced sample of observations. We discard the largest second-order increments in particular as well for the standardization. We illustrate the difference in Figure \ref{Fig:dataexample2} through the horizontal blue line which gives the threshold, or critical value, for the first performance of the test. As the plot reveals there are a few blue points below this line. Based on the thresholding variant these points would have been classified differently. Naturally, all these points are rather close to the threshold line. While the sequential application requires a more elaborate implementation, it has a more solid theoretical justification than the threshold variant. 

Finally, we apply estimator \eqref{Hest} to the time series to estimate the Hurst exponent. This estimator is usually applied within a parametric framework, but we can expect it to be robust with respect to time-varying volatility (of volatility) under sufficient regularity. 
We obtain $\hat H\approx 0{.}211$. Filtering out detected jumps with our above method yields a significantly higher estimate, $\hat H\approx 0{.}247$. Nevertheless, our results confirm the empirical evidence for rough volatility. Recalling the positive finite-sample bias due to jumps in our simulations and illustrated in Figure \ref{Fig:robust}, it seems to be surprising that $\hat H$ is increasing after removing jumps in the data example instead of decreasing. This observation can be explained by the stylized fact of the volatility outliers and their huge influence in this example. While a positive bias is expected from single jumps, since a jump inserted in the estimator \eqref{Hest} yields the value 1/2, an outlier, i.e., two neighbouring jumps of similar size in opposite directions, inserted in \eqref{Hest} yields the value 0. Therefore, volatility outliers result rather in a negative finite-sample bias. Instead, we obtain a smaller estimate of approx.\ 0{.}198 for $H$, when we discard the 50 largest increments which exceed a threshold of 1. Considering only the subsample from \cite{forecast}, we obtain an estimate of approx.\ 0{.}265 of the Hurst exponent what is in line with the ML-estimates in their Figure 6 (b). Since no jumps have been detected in this period, the estimate remains unchanged.

Overall, the empirical application underlines that it is important to take possible jumps into account for typical applications of rough fractional stochastic volatility. Moreover, we learn that outliers occur in these time series which can also be filtered out with our methods. 

\begin{figure}[t]
\begin{center}
\includegraphics[width=14.5cm]{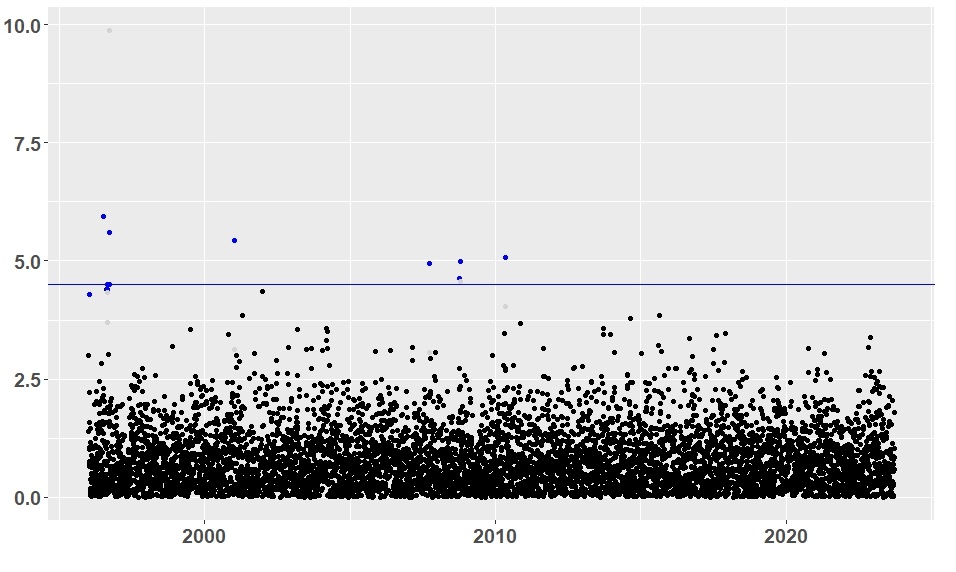}
\caption{Standardized second-order increments of volatility time series with detected jumps (blue) and threshold for the maximum.\label{Fig:dataexample2}}
\end{center}
\end{figure}

\section{Conclusion\label{sec:7}}
We solve the problem of testing for jumps based on high-frequency observations of a process with integral fractional part and present a Gumbel test with the desired asymptotic properties. We provide a localization result about the optimal estimation of jump times. This is new also for the semimartingale case, as such a result was not contained in the literature. It can be used to filter out jumps when the goal is inference on $H$ and $(\sigma_t^2)$. We point at good news for the rough volatility literature, that jumps should not manipulate their empirical finding that volatility is rough. Nevertheless, our empirical analysis reveals that jumps and outliers are practically relevant. It might be of interest for future research in economics to investigate how our findings affect forecasting results. Since the power of our approach increases with $H$, it works in particular well in the case when non-adjusted statistics for the continuous model without jumps become inconsistent. However, asymptotic results for using the suggested sequential top-down algorithm are readily feasible from our theory only when restricting to jumps of finite activity. For future research, it hence will be of interest to strengthen and generalize our results on the robustness of inference on $H$ and $(\sigma_t)$ with respect to jumps. For this purpose, it might be worth exploring approaches as multipower variations or a global jump filter, considered in \cite{multi} and \cite{inatsugu2021global} for Brownian semistationary processes and semimartingales, respectively. To model a volatility which is pre-estimated from observed prices, it would be of interest to further extend the methods to an observation model with additional noise. The Gumbel test for jump diffusion models has been extended to a noisy observation model in \cite{leemykland2}.

\section{Proofs\label{sec:8}}
\subsection{Groundwork from extreme value theory}\label{Appendix extreme value theory}
We require a result that the well-known Gumbel convergence for the maximum of i.i.d.\  standard normally distributed random variables generalizes to stationary sequences of weakly dependent normally distributed random variables.
\begin{lem}\label{Berman}
Let $(Y_n)_{n\in\mathbb{N}}$ be a standardised stationary sequence of normally distributed random variables, such that $\E[Y_n]=0$ and $\var(Y_n)=1$, with covariances $(\gamma(n))_{n\in\mathbb{N}}$ satisfying the condition $\lim_{n\to\infty}\log(n)\gamma(n)=0$. Then it holds that
\begin{eqnarray*}
\lim_{n\to\infty}\mathbb{P}\left(a_n\left( \max_{1\leq k\leq n}Y_k-b_n\right)\leq x \right)=\exp\left(-\exp(-x) \right),\quad x\in\mathbb{R},
\end{eqnarray*}
where $a_n$ and $b_n$ are given by
\begin{eqnarray*}
a_n=\sqrt{2\log(n)},\quad b_n=a_n-\frac{\log(\log( n))+\log(4\pi)}{2a_n}.
\end{eqnarray*}
 Furthermore if we define $c_n$ and $d_n$ by
\begin{eqnarray*}
c_n=\sqrt{2\log(2n)},\quad d_n=c_n-\frac{\log(\log( 2n))+\log(4\pi)}{2c_n},
\end{eqnarray*}
we obtain that
\begin{eqnarray*}
\lim_{n\to\infty}\mathbb{P}\left(c_n\left( \max_{1\leq k\leq n}\vert Y_k\vert-d_n\right)\leq x \right)=\exp\left(-\exp(-x) \right),\quad x\in\mathbb{R}.
\end{eqnarray*}
\end{lem}
The first part is Theorem 3.1 from \cite{berman1964limit}. The second part is readily implied by the symmetry of the standard normal distribution.

\begin{lem}[Rescaled Minima]\label{Minimum}
Let $(Y_n)_{n\geq 1}$ be a stationary sequence of normally distributed random variables with $\E[Y_n]=0$, and with
\begin{eqnarray*}
\lim_{n\to\infty}\mathbb{P}\left(a_n\left( \max_{1\leq k\leq n}Y_k-b_n\right)\leq x \right)=\exp(-\exp(-x)),\quad x\in\mathbb{R},
\end{eqnarray*}
where $a_n$ and $b_n$ are given by
\begin{eqnarray*}
a_n=\sqrt{2\log(n)},\quad b_n=a_n-\frac{\log(\log( n))+\log(4\pi)}{2a_n}.
\end{eqnarray*}
Then it holds as well that
\begin{eqnarray*}
\lim_{n\to\infty}\mathbb{P}\left(a_n\left( \min_{1\leq k\leq n}Y_k+b_n\right)\leq x \right)=1-\exp(-\exp(x)),\quad x\in\mathbb{R}.
\end{eqnarray*}
\end{lem}
\begin{proof}
By the symmetry of the normal distribution, the following equalities in distribution hold true:
\begin{eqnarray*}
\min_{1\leq k\leq n}Y_k\stackrel{d}{=}\min_{1\leq k\leq n}-Y_k\stackrel{d}{=}-\max_{1\leq k\leq n}Y_k.
\end{eqnarray*}
Therefore, we conclude that
\begin{eqnarray*}
\lim_{n\to\infty}\mathbb{P}\left(a_n\left( \min_{1\leq k\leq n}Y_k+b_n\right)\leq x \right)&=&1-\lim_{n\to\infty}\mathbb{P}\left(a_n\left( \max_{1\leq k\leq n}Y_k-b_n\right)\leq -x \right)\\
&=& 1-\exp(-\exp(x)).
\end{eqnarray*}
\end{proof}

\subsection{Asymptotic behaviour of the normalization and spot volatility estimator}
Since $(\sigma_s)_{s\in[0,1]}$ is continuous by Assumption \ref{Annahmen} on a compact interval, $(\sigma_s)_{s\in[0,1]}$ is also bounded, i.e.\ there exist constants $V$ and $K$, such that the upper bound in Assumption \ref{Annahmen} holds with constant $K$ and
\begin{align}0<V\leq \sigma_s\leq K,\quad 0\leq s\leq 1\,.\label{bounded}\end{align}
Throughout this section, $K_1$ and $K_2$ denote positive constants, which can change from line to line. $C_p$ denotes the $p$-th absolute moment of a standard normal distribution, $p\in(0,1]$. Set $ C_H=(4-2^{2H})$ for $H\in(0,1)$. This constant occurs in the variance of second-order increments of fractional Brownian motion, $\E[ (\Delta_{n,2}^{(2)}B^H)^2]= C_H n^{-2H}$. For $(h_n)_{n\in\mathbb{N}}$ a sequence of natural numbers with $\lim_{n\to\infty}h_n=\infty$, and $\lim_{n\to\infty}n^{-1}h_n=0$, we define for $p\in(0,1]$,   
    \begin{align}
    \label{sigmahat}\widehat{\sigma_{(k h_n)/n}^p}=n^{pH}h_n^{-1}C_p^{-1}\sum_{j=(k-1)h_n+2}^{kh_n} \Big\vert \Delta_{n,j}^{(2)}X \Big\vert^p,\quad k=1,\ldots, \lfloor n/h_n\rfloor.
    \end{align}
		We first consider the asymptotic behaviour of this normalization factor under the null hypothesis of no jumps. The statistic $\widehat{\sigma_{(k h_n)/n}^p}$ is in fact a consistent estimator for $C_H^{\frac{p}{2}}\sigma_{(k h_n)/n}^p$. We now determine bounds for the bias and the variance of this estimation. 
\begin{lem}[Bias]\label{Bias}
Assume that  $(\sigma_s)_{s\in[0,1]}$ satisfies Assumption \ref{Annahmen}. If $J_t= 0$, for every $t\in[0,1]$, it holds true that
\begin{eqnarray*}
\left(\mathbb{E}\left[\widehat{\sigma_{(k h_n)/n}^p}\right]-C_H^{\frac{p}{2}}\sigma_{(k h_n)/n}^p \right)^2=\mathcal{O}\left( n^{-2p\alpha}\right)+\mathcal{O}\Big(\Big(\frac{h_n}{n}\Big)^{2\alpha}\Big),\quad k=1,\ldots, \lfloor n/h_n\rfloor.
\end{eqnarray*}
\end{lem}
\begin{proof}
Define
\begin{eqnarray*}
A_n=\bigg(n^{pH}h_n^{-1}C_p^{-1}\sum_{j=(k-1)h_n+2}^{kh_n}\Big(\Big\vert \Delta_{n,j}^{(2)}X \Big\vert^p-\sigma_{(j-1)/n}^p\Big\vert \Delta_{n,j}^{(2)}B^H \Big\vert^p\Big) \bigg)
\end{eqnarray*}
and
\begin{eqnarray*}
B_n= \bigg(n^{pH}h_n^{-1}C_p^{-1}\sum_{j=(k-1)h_n+2}^{kh_n}\sigma_{(j-1)/n}^p\mathbb{E}\left[\Big\vert \Delta_{n,j}^{(2)}B^H \Big\vert^p \right]-C_H^{\frac{p}{2}}\sigma_{(k h_n)/n}^p \bigg).
\end{eqnarray*}
Using standard estimates, we see that it is enough to show that the following two inequalities hold:
\begin{eqnarray*}
\mathbb{E}[A_n^2]\leq K_1 n^{-2p\alpha}\quad \text{ and } \quad B_n^2\leq K_2 \left( \frac{h_n}{n}\right)^{2\alpha}.
\end{eqnarray*}
Writing second-order increments as the differences of first-order increments, a standard inequality yields that
\begin{eqnarray*}
A_n&\leq& n^{pH}h_n^{-1}C_p^{-1}\sum_{j=(k-1)h_n+2}^{kh_n} \Big\vert \Delta_{n,j}^{(1)}X-\sigma_{(j-1)/n}\Delta_{n,j}^{(1)}B^H \Big\vert^p\\
 &+&  n^{pH}h_n^{-1}C_p^{-1}\sum_{j=(k-1)h_n+2}^{kh_n} \Big\vert \Delta_{n,j-1}^{(1)}X-\sigma_{(j-1)/n}\Delta_{n,j-1}^{(1)}B^H \Big\vert^p\,.
\end{eqnarray*}
Now, the Love-Young inequality, see Theorem 1.16 in \cite{kubilius2017parameter}, implies that
\begin{eqnarray*}
   \Big\vert \Delta_{n,j}^{(1)}X-\sigma_{(j-1)/n}\Delta_{n,j}^{(1)}B^H \Big\vert^p\leq  C(p,\alpha,H) n^{-p(\alpha+H)}
   \end{eqnarray*}
 uniformly for all $j$, and analogously with the shifted increments that
   \begin{eqnarray*}
    \Big\vert \Delta_{n,j-1}^{(1)}X-\sigma_{(j-1)/n}\Delta_{n,j-1}^{(1)}B^H \Big\vert^p\leq C(p,\alpha,H) n^{-p(\alpha+H)},
\end{eqnarray*}
where $C(p,\alpha,H)$ is a positive constant, which comes from the H\"older continuity of the fractional Brownian motion and $(\sigma_s)_{s\in[0,1]}$. Note that the constant coming from the H\"older continuity of $(\sigma_s)_{s\in[0,1]}$ is bounded by Assumption \ref{Annahmen} and that, by Theorem 1 of \cite{azmoodeh2014necessary}, the constant coming from the H\"older continuity of the fractional Brownian motion has moments of all order. Thus, we conclude that 
\begin{eqnarray*}
   \mathbb{E}[A_n^2]&\leq& 2\mathbb{E}\left[\left(  n^{pH}h_n^{-1}C_p^{-1}\sum_{j=(k-1)h_n+2}^{kh_n} \Big\vert \Delta_{n,j}^{(1)}X-\sigma_{(j-1)/n}\Delta_{n,j}^{(1)}B^H \Big\vert^p   \right)^2  \right]\\
   &+& 2\mathbb{E}\left[\left( n^{pH}h_n^{-1}C_p^{-1}\sum_{j=(k-1)h_n+2}^{kh_n} \Big\vert \Delta_{n,j-1}^{(1)}X-\sigma_{(j-1)/n}\Delta_{n,j-1}^{(1)}B^H \Big\vert^p \right)^2  \right]\leq K_1  n^{-2p\alpha}.
\end{eqnarray*} 
Applying the Cauchy-Schwarz inequality and the H\"older continuity of $(\sigma_s^p)_{s\in[0,1]}$,  yields that
\begin{eqnarray*}
B_n^2&=&\bigg(C_H^{\frac{p}{2}}h_n^{-1}\sum_{j=(k-1)h_n+2}^{kh_n}\left(\sigma_{(j-1)/n}^p - \sigma_{(k h_n)/n}^p\right) \bigg)^2\\
&\leq& C_H^{p} h_n^{-1} \sum_{j=(k-1)h_n+1}^{kh_n}\left(\sigma_{(j-1)/n}^p - \sigma_{(k h_n)/n}^p\right)^2\\
&\leq& K_2 h_n^{-1}   \sum_{j=(k-1)h_n+1}^{kh_n}\left(\frac{k h_n}{n}-\frac{j-1}{n}\right)^{2\alpha}\\
&\leq & K_2 \left(\frac{h_n}{n}\right)^{2\alpha}.
\end{eqnarray*}
\end{proof}
\begin{lem}[Variance]\label{Varianz}
Assume that  $(\sigma_s)_{s\in[0,1]}$ satisfies Assumption \ref{Annahmen}. If $J_t= 0$, for every $t\in[0,1]$, it holds true that
\begin{eqnarray*}
\mathbb{V}\text{ar}\left( \widehat{\sigma_{(k h_n)/n}^p} \right)=\mathcal{O}(n^{-2p\alpha})+\mathcal{O}(h_n^{-1}),\quad k=1,\ldots, \lfloor n/h_n\rfloor .
\end{eqnarray*}
\end{lem}
\begin{proof}
Consider $A_n$ from the previous proof and define
\begin{eqnarray*}
C_n= \bigg(n^{pH}h_n^{-1}C_p^{-1}\sum_{j=(k-1)h_n+2}^{kh_n}\sigma_{(j-1)/n}^p\left(\Big\vert \Delta_{n,j}^{(2)}B^H \Big\vert^p-\mathbb{E}\left[\Big\vert \Delta_{n,j}^{(2)}B^H \Big\vert^p \right] \right)\bigg).
\end{eqnarray*}
Using standard estimates, we see that it is enough to show that the following two inequalities hold:
\begin{eqnarray*}
\mathbb{E}[A_n^2]\leq K_1 n^{-2p\alpha} \quad \text{ and } \quad \mathbb{E}[C_n^2]\leq K_2 h_n^{-1}.
\end{eqnarray*}
The first one is established in the previous proof. We are left to prove the second one. In the sequel, we will exploit the self-similarity of fractional Brownian motion for a convenient notation and therefore consider our fractional Brownian motion $B^H$ defined for all times $t\in[0,\infty)$. In line with our notation for second-order increments from Section \ref{subsection 2.1}, we write
\begin{align} \label{ss}\Delta_{1,j}^{(2)}B^H=B^H_{j}-2B^H_{j-1} + B^H_{j-2}~\mbox{, for}~j\in\N.\end{align} 
Next, define 
\begin{eqnarray*}
\rho(k)=\mathbb{C}\text{ov}\left(\Delta_{1,j+k}^{(2)}B^H, \Delta_{1,j}^{(2)}B^H  \right),~\rho(-k)=\rho(k), \quad\text{ for } k\in\N,
\end{eqnarray*}
 and then use a minor modification of Proposition 5.2.4 from \cite{pipiras2017long}, to obtain that
 \begin{eqnarray*}
\mathbb{E}[C_n^2]&=& C_p^{-2} h_n^{-2}\sum_{j=2}^{h_n}\sum_{i=2}^{h_n}\sigma_{(j-1)/n}^p\sigma_{(i-1)/n}^p\mathbb{C}\text{ov}\left( \Big\vert \Delta_{1,j}^{(2)}B^H \Big\vert^p,\Big\vert \Delta_{1,i}^{(2)}B^H \Big\vert^p \right)\\
&\leq& K_2 C_p^{-2} h_n^{-2}\sum_{j=2}^{h_n}\sum_{i=2}^{h_n}\sigma_{(j-1)/n}^p\sigma_{(i-1)/n}^p (\rho( j-i))^2\\
&\leq & K_2 C_p^{-2} K^{2p} h_n^{-2}\sum_{k=-(h_n-2)}^{h_n-2}(h_n-\vert k \vert)(\rho(k))^2\\
&\leq & K_2 C_p^{-2} K^{2p} h_n^{-1}\sum_{k\in\mathbb{Z}}(\rho(k))^2\,.
\end{eqnarray*}
Since $(\rho(n))^2\propto n^{4H-8}$, see, for instance, Lemma 1 in \cite{coeurjolly2001estimating}, the last sum is finite and we get the desired result.
\end{proof}
The bias-variance decomposition readily allows finding an optimal choice of $h_n$, which minimises the mean squared error. 
\begin{cor} 
Assume that $(\sigma_s)_{s\in[0,1]}$ satisfies Assumption \ref{Annahmen}. If we further assume that $h_n$ satisfies the condition $h_n\propto n^{2\alpha/(2\alpha+1)}$ and that $J_t= 0$, for every $t\in[0,1]$, it holds true that
\begin{eqnarray*}
    \mathbb{V}\text{ar}\left( \widehat{\sigma_{(k h_n)/n}^p} \right)+\left(\mathbb{E}\left[\widehat{\sigma_{(k h_n)/n}^p}\right]-C_H^{\frac{p}{2}}\sigma_{(k h_n)/n}^p \right)^2 =\mathcal{O}\left(n^{-\left(\frac{2\alpha}{2\alpha+1} \wedge p\alpha\right)}\right),\quad k=1,\ldots, \lfloor n/h_n\rfloor.
    \end{eqnarray*}
\end{cor}
To prepare the proof of consistency of our test under the alternative hypothesis, we have to show that the normalization with the estimated spot volatility is sufficiently robust with respect to the additive jump component $(J_t)$. This will be formalised in the next lemma.

\begin{lem}\label{Robustheit}
Suppose Assumption \ref{Annahmen} and that Assumption \ref{AnnahmeSpruenge} is satisfied for some $p \in (0,1]$. For $h_n\propto n^{\beta}$, with $\beta\in (pH,1)$, it holds that
    \begin{eqnarray*}
\max_{k}\bigg\vert n^{pH}\sum_{j=(k-1)h_n+2}^{kh_n}\frac{\left(\Big\vert \Delta_{n,j}^{(2)}X\Big\vert^p-\Big\vert \Delta_{n,j}^{(2)}Y \Big\vert^p\right)}{h_nC_p}\bigg\vert=\KLEINO_{\mathbb{P}}(n^{-\gamma}) \quad \text{for}\quad 0<\gamma<\beta-pH.
\end{eqnarray*}
\end{lem}
\begin{proof}
Under the stated conditions, we have that $\lim_{n\to\infty}n^{pH}h_n^{-1}=0$. Therefore, it is enough to show that
\begin{eqnarray*}
\bigg\vert \Big\vert \Delta_{n,j}^{(2)}X\Big\vert^p-\Big\vert \Delta_{n,j}^{(2)}Y \Big\vert^p\bigg\vert
\leq \Big\vert \Delta_{n,j}^{(1)}J\Big\vert^p+\Big\vert \Delta_{n,j-1}^{(1)}J \Big\vert^p.
\end{eqnarray*}
Define $A_j=\left\lbrace \Big\vert  \Delta_{n,j}^{(2)}X \Big\vert^p> \Big\vert \Delta_{n,j}^{(2)}Y \Big\vert^p   \right\rbrace$
and note that 
\begin{eqnarray*}
\Big\vert\Big\vert \Delta_{n,j}^{(2)}X \Big\vert^p-\Big\vert \Delta_{n,j}^{(2)}Y \Big\vert^p\Big\vert\cdot (\mathbbm{1}_{A_j}+\mathbbm{1}_{A_j^c})
\leq\Big\vert  \Delta_{n,j}^{(2)}J \Big\vert^p,
\end{eqnarray*}
where the last inequality follows from the fact that $f(x)=\vert x \vert^p$ is subadditive with the two inequalities
\begin{eqnarray*}
&&\Big\vert  \Delta_{n,j}^{(2)}X \Big\vert^p\leq \Big\vert  \Delta_{n,j}^{(2)}Y \Big\vert^p+\Big\vert  \Delta_{n,j}^{(2)}J \Big\vert^p \quad \text{ and } \quad 
\Big\vert  \Delta_{n,j}^{(2)}Y \Big\vert^p-\Big\vert  \Delta_{n,j}^{(2)}J \Big\vert^p\leq \Big\vert  \Delta_{n,j}^{(2)}X \Big\vert^p.
\end{eqnarray*}
Applying the subadditivity to $| \Delta_{n,j}^{(2)}J |^p$ yields the result.
\end{proof}
For the proof of the Gumbel convergence, we will use uniform consistency of the spot volatility estimation, or that the normalising factors converge uniformly in probability, respectively. The next lemma will clarify this statement.

\begin{lem}\label{GK}
Assume that $(\sigma_t)_{t\in[0,1]}$ satisfies Assumption \ref{Annahmen}.
\begin{enumerate}
    \item[(i)] Let $p\in(0,1]$ and set $h_n\propto n^{\beta}$, for $\beta\in (pH,1)$. If $J_t=0$, for all $t\in[0,1]$, it holds for $0<\beta_1<\beta/2$ and $0<\beta_2<p\alpha$, that
 \begin{eqnarray*}
\max_{k}\Big\vert \widehat{\sigma_{(k h_n)/n}^p}-C_H^{\frac{p}{2}}\sigma_{(k h_n)/n}^p \Big\vert=\KLEINO_{\mathbb{P}}(n^{-(\beta_1\wedge \beta_2)}).
\end{eqnarray*}
    \item[(ii)] Suppose that  Assumption \ref{AnnahmeSpruenge} is satisfied for some $p \in (0,1]$ and select $h_n\propto n^{\beta}$, for $\beta\in (pH,1)$. For $0<\beta_1<\beta/2$, $0<\beta_2<p\alpha$, and $0<\beta_3<\beta-pH$, it holds that
    
\begin{eqnarray*}
\max_{k}\Big\vert \widehat{\sigma_{(k h_n)/n}^p}-C_H^{\frac{p}{2}}\sigma_{(k h_n)/n}^p \Big\vert=\KLEINO_{\mathbb{P}}(n^{-(\beta_1\wedge \beta_2\wedge \beta_3)}).
\end{eqnarray*}
\end{enumerate}
\end{lem}
\begin{proof}
We will prove the first part of the lemma. The second part can be shown analogously. The application of standard inequalities for maxima and the reverse triangle inequality yield the following estimate
\begin{eqnarray*}
&&\max_{k}\Big\vert  n^{pH}h_n^{-1}C_p^{-1}\sum_{j=(k-1)h_n+2}^{kh_n} \Big\vert \Delta_{n,j}^{(2)}X \Big\vert^p -C_H^{\frac{p}{2}}\sigma_{(kh_n)/n}^p  \Big\vert\\
&\leq& \max_{k} n^{pH}h_n^{-1}C_p^{-1}\sum_{j=(k-1)h_n+2}^{kh_n}\Big\vert \Delta_{n,j}^{(1)}X -\sigma_{(j-1)/n} \left( \Delta_{n,j}^{(1)}B^H\right)\Big\vert^p \\
&+& \max_{k} n^{pH}h_n^{-1}C_p^{-1}\sum_{j=(k-1)h_n+2}^{kh_n} \Big\vert \Delta_{n,j-1}^{(1)}X -\sigma_{(j-1)/n} \left( \Delta_{n,j-1}^{(2)}B^H\right)\Big\vert^p \\
&+&\max_{k}\Big\vert  n^{pH}h_n^{-1}C_p^{-1}\sum_{j=(k-1)h_n+2}^{kh_n} \sigma_{(j-1)/n}^p\Big\vert \Delta_{n,j}^{(2)}B^H \Big\vert^p -C_H^{\frac{p}{2}}\sigma_{(kh_n)/n}^p \Big\vert.
\end{eqnarray*}
By the Love-Young inequality, the first two addends are of order $\mathcal{O}_{\mathbb{P}}(n^{-p\alpha})=\KLEINO_{\mathbb{P}}(n^{-\beta_2})$. Recall the notation from \eqref{ss}. Since $(\sigma_s^p)_{s\in[0,1]}$ is H\"older continuous and bounded with the constant $K$ from \eqref{bounded}, we conclude the following inequality:
\begin{align*}
&\max_{k}\Big\vert  n^{pH}h_n^{-1}C_p^{-1}\sum_{j=(k-1)h_n+2}^{kh_n}\sigma_{(j-1)/n}^p \Big\vert \Delta_{n,j}^{(2)}B^H\Big\vert^p -C_H^{\frac{p}{2}}\sigma_{(kh_n)/n}^p    \Big\vert\\
\leq& K_2\left(\frac{h_n}{n}\right)^{\alpha}   h_n^{-1}C_p^{-1} \max_{k} \sum_{j=(k-1)h_n+2}^{kh_n}\Big\vert \Delta_{1,j}^{(2)}B^H\Big\vert^p +K^p\max_{k}\Big\vert   h_n^{-1}C_p^{-1}\sum_{j=(k-1)h_n+2}^{kh_n} \Big\vert \Delta_{1,j}^{(2)}B^H \Big\vert^p-C_H^{\frac{p}{2}}\Big\vert.
\end{align*}
We show that the second addend is of order $\KLEINO_{\mathbb{P}}(n^{-\beta_1})$ what finishes the proof. Applying Markov's inequality with $f:\mathbb{R}_{+}\to \mathbb{R}_{+}, x\mapsto x^q$, we obtain for $\varepsilon>0$ that
\begin{eqnarray*}
&&\mathbb{P}\left(\Big\vert h_n^{-1}\sum_{k=2}^{h_n}C_p^{-1}\vert  \Delta_{1,k}^{(2)}B^H \vert^p-C_H^{\frac{p}{2}}\Big\vert\geq \varepsilon \right)
\leq \frac{\mathbb{E}\left[\left( \sum_{k=2}^{h_n}  (\vert\Delta_{1,k}^{(2)}B^H\vert^p-C_pC_H^{\frac{p}{2}})  \right)^q  \right]   }{h_n^qC_p^{q}\varepsilon^q}.
\end{eqnarray*}
The time series $ (\vert \Delta_{1,j}^{(2)}B^H \vert^p-C_pC_H^{\frac{p}{2}})_{j\geq 2}$ is stationary, see Chapter 2 in \cite{coeurjolly2001estimating}. Set $\rho(k)=\mathbb{C}\text{ov}( \vert \Delta_{1,2}^{(2)}B^H\vert^p,\vert \Delta_{1,2+k}^{(2)}B^H\vert^p)$, for $k\in\mathbb{N}_0$. The Rosenthal-type inequality for stationary, weakly dependent time series in Proposition 21 from \cite{merlevede2013rosenthal} applied to $ (\vert \Delta_{1,j}^{(2)}B^H \vert^p-C_pC_H^{\frac{p}{2}})_{j\geq 2}$ implies that
\begin{align*}
\mathbb{P}\left(\Big\vert h_n^{-1}\sum_{k=2}^{h_n}C_p^{-1}\vert  \Delta_{1,k}^{(2)}B^H \vert^p-C_H^{\frac{p}{2}}\Big\vert\geq \varepsilon \right)\leq \frac{h_n^{q/2}\left(\sum_{k=0}^{\infty}\rho(k) \right)^{q/2}+h_n \E\big[\vert \Delta_{1,2}^{(2)}B^H\vert^{pq}\big]}{h_n^q C_p^q \varepsilon^q}.
\end{align*}
For any $q\ge 2$, this yields that
\begin{align}\label{rosenthal}
\mathbb{P}\left(\Big\vert h_n^{-1}\sum_{k=2}^{h_n}C_p^{-1}\vert \Delta_{1,k}^{(2)}B^H\vert^p-C_H^{\frac{p}{2}}     \Big\vert\geq \varepsilon        \right)=\mathcal{O}\big(h_n^{-\frac{q}{2}}\big).
\end{align}
We hence obtain for any $\varepsilon>0$ that
\begin{align*}
  &\mathbb{P}\Big(K^p\max_{k}\Big\vert   h_n^{-1}C_p^{-1}\sum_{j=(k-1)h_n+2}^{kh_n} \Big\vert \Delta_{1,j}^{(2)}B^H \Big\vert^p-C_H^{\frac{p}{2}}\Big\vert>\varepsilon)\\
	&\le \frac{n}{h_n}\cdot \mathbb{P}\Big(K^p C_p^{-1} h_n^{-1}\Big\vert \sum_{j=2}^{h_n} \Big\vert \Delta_{1,j}^{(2)}B^H \Big\vert^p-C_H^{\frac{p}{2}}\Big\vert>\varepsilon\Big)\\
	&=\mathcal{O}\big(n^{1-\beta} n^{-q\beta/2}\big),
\end{align*}
using standard estimates and \eqref{rosenthal}. Setting $q\in\N$ sufficiently large, we obtain for any $\varepsilon>0$ that 
\begin{align*}
\mathbb{P}\Big(n^{\beta_1}K^p\max_{k}\Big\vert   h_n^{-1}C_p^{-1}\sum_{j=(k-1)h_n+2}^{kh_n} \Big\vert \Delta_{1,j}^{(2)}B^H \Big\vert^p-C_H^{\frac{p}{2}}\Big\vert>\varepsilon)=\mathcal{O}\big(n^{1-\beta} n^{q(\beta_1-\beta/2)}\big)=\KLEINO(1),
\end{align*} 
since $\beta_1<\beta/2$. 
\end{proof}
\subsection{Asymptotic behaviour of the test statistics}
The next lemma is crucial. It establishes that our maxima statistics on which the tests are based on have the same limit distribution as the maximum of the normalised (absolute) second-order increments of fractional Brownian motion.  
\begin{lem}\label{Approximation}
Assume that $(\sigma_t)_{t\in[0,1]}$ satisfies Assumption  \ref{Annahmen}.

\begin{enumerate}
    \item[(i)]   Let $p\in(0,1]$ and set $h_n\propto n^{\beta}$, for $\beta\in (pH,1)$. If $J_t=0$, for all $t\in[0,1]$, it holds for $0<\beta_1<\beta/2$ and $0<\beta_2<p\alpha$, that
    \begin{eqnarray*}
    &&\max_{k,j}\Bigg\vert \left(\widehat{\sigma_{(k h_n)/n}^p}\right)^{-\frac{1}{p}} \Delta_{n,j}^{(2)}X -C_H^{-\frac{1}{2}}\Delta_{n,j}^{(2)}B^H \Bigg\vert
= \KLEINO_{\mathbb{P}}\left(n^{-(\beta_1\wedge \beta_2)-H}\sqrt{2\log(2n)}\right).
    \end{eqnarray*}
    \item[(ii)] Suppose that  Assumption \ref{AnnahmeSpruenge} is satisfied for some $p \in (0,1]$, and select $h_n\propto n^{\beta}$, with $\beta\in (pH,1)$. For $0<\beta_1<\beta/2$, $0<\beta_2<p\alpha$, and $0<\beta_3<\beta-pH$, it holds that
\begin{eqnarray*}
&&\max_{k,j}\Bigg\vert \left(\widehat{\sigma_{(k h_n)/n}^p}\right)^{-\frac{1}{p}} \Delta_{n,j}^{(2)}Y -C_H^{-\frac{1}{2}}\Delta_{n,j}^{(2)}B^H \Bigg\vert
= \KLEINO_{\mathbb{P}}\left(n^{-(\beta_1\wedge \beta_2\wedge\beta_3)-H}\sqrt{2\log(2n)}\right).
\end{eqnarray*}
\end{enumerate}
\end{lem}
\begin{proof}
We will prove the first part of the lemma. The second part can be shown analogously. Recall that $c_n=\sqrt{2\log(2n)}$. Standard estimates for maxima yield that
\begin{eqnarray*}
\max_{k,j}\Bigg\vert \left(\widehat{\sigma_{(k h_n)/n}^p}\right)^{-\frac{1}{p}} \Delta_{n,j}^{(2)}X
-C_H^{-\frac{1}{2}}\frac{\sigma_{(k h_n)/n} \Delta_{n,j}^{(2)}B^H}{ \sigma_{(k h_n)/n}} \Bigg\vert
&\leq&  \max_{k,j}\Bigg\vert \left(\widehat{\sigma_{(k h_n)/n}^p}\right)^{-\frac{1}{p}} \Delta_{n,j}^{(2)}X- C_H^{-\frac{1}{2}}\frac{\Delta_{n,j}^{(2)}X}{\sigma_{(k h_n)/n}}  \Bigg\vert\\
&+& V^{-1}C_H^{-\frac{1}{2}} \max_{k,j}\Big\vert \Delta_{n,j}^{(2)}X-\sigma_{(kh_n)/n}\Delta_{n,j}^{(2)}B^H \Big\vert,
\end{eqnarray*}
with $V$ from \eqref{bounded}. We first deal with the second addend. Note that we have the following inequality: 
\begin{eqnarray*}
\Big\vert \left(\Delta_{n,j}^{(2)}X\right)-\sigma_{(kh_n)/n}\Delta_{n,j}^{(2)}B^H \Big\vert
&\leq& \Big\vert \Delta_{n,j}^{(1)}X-\sigma_{(kh_n)/n}\Delta_{n,j}^{(1)}B^H \Big\vert
+\Big\vert \Delta_{n,j-1}^{(1)}X-\sigma_{(kh_n)/n}\Delta_{n,j-1}^{(1)}B^H \Big\vert\,,
\end{eqnarray*}
as well as the identity
\begin{eqnarray*}
\Delta_{n,j}^{(1)}X-\sigma_{(kh_n)/n}\Delta_{n,j}^{(1)}B^H
= \Delta_{n,j}^{(1)}X-\sigma_{(j-1)/n}\Delta_{n,j}^{(1)}B^H+(\sigma_{(j-1)/n}-\sigma_{(kh_n)/n})\Delta_{n,j}^{(1)}B^H.
\end{eqnarray*}
A direct application of the Love-Young inequality yields that
\begin{eqnarray*}
\max_{k,j}\Big\vert \Delta_{n,j}^{(1)}X-\sigma_{(j-1)/n}\Delta_{n,j}^{(2)}B^H\Big\vert=\mathcal{O}_{\mathbb{P}}(n^{-\alpha-H}). 
\end{eqnarray*}
Since $(\sigma_s)_{s\in[0,1]}$ is H\"older continuous, we obtain that
\begin{eqnarray*}
\max_{k,j}\Big\vert (\sigma_{(j-1)/n}-\sigma_{(kh_n)/n})\Delta_{n,j}^{(1)}B^H \Big\vert\leq K_1\left(\frac{h_n}{n} \right)^{\alpha}\max_{1\leq j\leq n}\Big\vert \Delta_{n,j}^{(1)}B^H \Big\vert
=\mathcal{O}_{\mathbb{P}}\left( \left(\frac{h_n}{n} \right)^{\alpha}n^{-H}c_n \right).
\end{eqnarray*}
As a result, we conclude that 
\begin{eqnarray*}
  V^{-1}C_H^{-\frac{1}{2}} \max_{k,j}\Bigg\vert \Delta_{n,j}^{(2)}X-\sigma_{(kh_n)/n}\Delta_{n,j}^{(2)}B^H \Bigg\vert= \mathcal{O}_{\mathbb{P}}(n^{-\alpha-H})+\mathcal{O}_{\mathbb{P}}\left( \left(\frac{h_n}{n} \right)^{\alpha}n^{-H}c_n \right). 
\end{eqnarray*}
We now deal with the first addend above. We obtain the following estimate:
\begin{eqnarray*}
 \max_{k,j}\Bigg\vert \left(\widehat{\sigma_{(k h_n)/n}^p}\right)^{-\frac{1}{p}} \Delta_{n,j}^{(2)}X-C_H^{-\frac{1}{2}} \frac{\Delta_{n,j}^{(2)}X}{\sigma_{(k h_n)/n}}  \Bigg\vert
 \leq \frac{\max_{k}\Big\vert C_H^{\frac{1}{2}}\sigma_{(k h_n)/n}-\left(\widehat{\sigma_{(k h_n)/n}^p}\right)^{\frac{1}{p}} \Big\vert}{\min_{k}\Big\vert \sigma_{(k h_n)/n}\cdot  \left(\widehat{\sigma_{(k h_n)/n}^p}\right)^{\frac{1}{p}}C_H^{\frac{1}{2}}\Big\vert}\cdot \max_{k,j}\left(  \Big\vert  \Delta_{n,j}^{(2)}X\Big\vert \right)\,.
\end{eqnarray*}
Since $f(x)=x^{1/p}$ is  Lipschitz continuous on a compact interval and $(\sigma_s)_{s\in[0,1]}$ is bounded, we obtain the following upper bound: 
\begin{eqnarray*}
\max_{k}\Big\vert C_H^{\frac{1}{2}} \sigma_{(k h_n)/n}\cdot\left(\widehat{\sigma_{(k h_n)/n}^p}\right)^{\frac{1}{p}}- C_H\sigma_{(k h_n)/n}^2\Big\vert&\leq& KC_H^{\frac{1}{2}} \max_{1\leq k\leq m}\Big\vert \left(\widehat{\sigma_{(k h_n)/n}^p}\right)^{\frac{1}{p}}- C_H^{\frac{1}{2}}\sigma_{(k h_n)/n}\Big\vert\\
&\leq& K_1 C_H^{\frac{1}{2}} \max_{k}\Big\vert \widehat{\sigma_{(k h_n)/n}^p}- C_H^{\frac{p}{2}}\sigma_{(k h_n)/n}^p\Big\vert\\
&=&\KLEINO_{\mathbb{P}}\left(n^{-(\beta_1\wedge \beta_2)}\right),
\end{eqnarray*}
where we applied the first part of Lemma \ref{GK} in the final step. This bound together with
\begin{eqnarray*}
&&\min_{k}\Big\vert  \sigma_{(k h_n)/n}\cdot\left(\widehat{\sigma_{(k h_n)/n}^p}\right)^{\frac{1}{p}}C_H^{\frac{1}{2}} \Big\vert\\
&\geq& \min_{k}\left( C_H\sigma_{(k h_n)/n}^2-\Big\vert C_H^{\frac{1}{2}} \sigma_{(k h_n)/n}\cdot\left(\widehat{\sigma_{(k h_n)/n}^p}\right)^{\frac{1}{p}}- C_H\sigma_{(k h_n)/n}^2\Big\vert\right)\\
&\geq& C_H V^2-\max_{k} \Big\vert  C_H^{\frac{1}{2}}\sigma_{(k h_n)/n}\cdot\left(\widehat{\sigma_{(k h_n)/n}^p}\right)^{\frac{1}{p}}-C_H \sigma_{(k h_n)/n}^2\Big\vert,
\end{eqnarray*}
implies that 
\begin{eqnarray*}
\frac{\max_{k}\Big\vert C_H^{\frac{1}{2}}\sigma_{(k h_n)/n}-\left(\widehat{\sigma_{(k h_n)/n}^p}\right)^{\frac{1}{p}} \Big\vert}{\min_{k}\Big\vert \sigma_{(k h_n)/n}\cdot  \left(\widehat{\sigma_{(k h_n)/n}^p}\right)^{\frac{1}{p}}C_H^{\frac{1}{2}}\Big\vert}=\KLEINO_{\mathbb{P}}(n^{-(\beta_1\wedge \beta_2)}).
\end{eqnarray*}
Since
\begin{eqnarray*}
\max_{k,j}  \Big\vert  \Delta_{n,j}^{(2)}X\Big\vert \leq \mathcal{O}_{\mathbb{P}}(n^{-\alpha-H})+K\max_{1\leq j\leq n}\vert \Delta_{n,j}^{(2)}B^H \vert,
\end{eqnarray*}
we obtain that 
\begin{eqnarray*}
\max_{k,j}\Bigg\vert \left(\widehat{\sigma_{(k h_n)/n}^p}\right)^{-\frac{1}{p}} \Delta_{n,j}^{(2)}X-C_H^{-\frac{1}{2}} \frac{\Delta_{n,j}^{(2)}X}{\sigma_{(k h_n)/n}}  \Bigg\vert
&=& \KLEINO_{\mathbb{P}}(n^{-(\beta_1\wedge \beta_2)-\alpha-H})+\KLEINO_{\mathbb{P}}(n^{-(\beta_1\wedge \beta_2)-H}c_n)\\
&=&\KLEINO_{\mathbb{P}}(n^{-(\beta_1\wedge \beta_2)-H}c_n)\,,
\end{eqnarray*}
and conclude
\begin{eqnarray*}
  &&  \max_{k,j}\Bigg\vert \left(\widehat{\sigma_{(k h_n)/n}^p}\right)^{-\frac{1}{p}} \Delta_{n,j}^{(2)}X
-C_H^{-\frac{1}{2}}\frac{\sigma_{(k h_n)/n} \Delta_{n,j}^{(2)}B^H}{ \sigma_{(k h_n)/n}} \Bigg\vert\\
&=&\KLEINO_{\mathbb{P}}(n^{-(\beta_1\wedge \beta_2)-H}c_n)+\mathcal{O}_{\mathbb{P}}(n^{-\alpha-H})+\mathcal{O}_{\mathbb{P}}\left(\left(\frac{h_n}{n} \right)^{\alpha}n^{-H}c_n \right)\\
&=& \KLEINO_{\mathbb{P}}(n^{-(\beta_1\wedge \beta_2)-H}c_n).
\end{eqnarray*}
\end{proof}
\textit{Proof of Theorem \ref{Gumbelkonvergenz}:}
We will only prove the first part of the theorem. The proof of the second part follows analogously. Note that the claim follows from Lemma \ref{Berman}, if we can show that
\begin{eqnarray*}
    \Bigg\vert  a_n(T_n-b_n)-a_n\left(\max_{1\leq j\leq n}n^HC_H^{-\frac{1}{2}}\Delta_{n,j}^{(2)}B^H -b_n\right)   \Bigg\vert=\KLEINO_{\mathbb{P}}(1).
    \end{eqnarray*}
   Applying the inequality
   \begin{eqnarray*}
   \vert \max_{1\leq k\leq n}(x_k+y_k)-\max_{1\leq k\leq n}x_k  \vert\leq \max_{1\leq k\leq n}\vert y_k \vert, \quad x_1,\ldots,x_n,y_1,\ldots,y_n\in\mathbb{R},
   \end{eqnarray*}
   we obtain the estimate
   \begin{eqnarray*}
   \Big\vert  a_n(T_n-b_n)-a_n\left(\max_{1\leq j\leq n}n^HC_H^{-\frac{1}{2}}\Delta_{n,j}^{(2)}B^H -b_n\right)   \Big\vert
   \leq a_nn^H \max_{k,j}\Big\vert \left(\widehat{\sigma_{(k h_n)/n}^p}\right)^{-\frac{1}{p}} \Delta_{n,j}^{(2)}X -C_H^{-\frac{1}{2}}\Delta_{n,j}^{(2)}B^H \Big\vert.
   \end{eqnarray*}
   The upper bound is $\KLEINO_{\mathbb{P}}(1)$, due to the first part of Lemma  \ref{Approximation}, and the claim follows.\hfill\qed
	
\vspace*{.2cm}\noindent
\textit{Proof of Theorem \ref{Divergenz der Teststatistik}:} 
We only prove the first part $(i)$ of the theorem. The proof of the second part $(ii)$ is completely analogous.

First, recall that we discretely observe the mixture process $X_t=Y_t+J_t$, $t\in[0,1]$. Since for arbitrary real numbers $x_1,\ldots,x_n,y_1,\ldots,y_n\in\mathbb{R}$, it holds that  
    \begin{eqnarray*}
     \max_{1\leq k\leq n}(x_k+y_k)\geq \max_{1\leq k\leq n}x_k+\min_{1\leq k\leq n}y_k,
    \end{eqnarray*}
    we can use the following lower bound
\begin{eqnarray*}
T_n
\geq \max_{k,j} \left(\widehat{\sigma_{(k h_n)/n}^p}\right)^{-\frac{1}{p}}n^H\Delta_{n,j}^{(2)}J
+ \min_{k,j} \left(\widehat{\sigma_{(k h_n)/n}^p}\right)^{-\frac{1}{p}}n^H \Delta_{n,j}^{(2)}Y.
\end{eqnarray*}
Using the previous bound for $T_n$, we conclude that 
\begin{eqnarray*}
a_n(T_n-b_n)
\geq a_n n^H\max_{k,j} \left(\widehat{\sigma_{(k h_n)/n}^p}\right)^{-\frac{1}{p}}\Delta_{n,j}^{(2)}J
+ a_n\left(  \min_{k,j} \left(\widehat{\sigma_{(k h_n)/n}^p}\right)^{-\frac{1}{p}}n^H\Delta_{n,j}^{(2)}Y+b_n\right)-2a_nb_n.\\
\end{eqnarray*} 
Due to the second part of Lemma \ref{Approximation}, and Lemma \ref{Minimum}, the second addend has the same limit law as 
\begin{eqnarray*}
a_n\left( n^H C_H^{-\frac{1}{2}} \min_{2\leq j\leq n}\Delta_{n,j}^{(2)}B^H  +b_n\right)
\end{eqnarray*}
and is therefore bounded in probability. By our assumption, there exists $\omega\in\Omega_1$ and $t\in(0,1]$, with $\Delta J_t(\omega):=\lim_{s\nearrow t }(J_t(\omega)-J_s(\omega))>0$. Therefore, we conclude that 
\begin{eqnarray*}
\liminf_{n\to\infty}\max_{2\leq j\leq n}\Delta_{n,j}^{(2)}J(\omega)
\geq \liminf_{n\to\infty}\Delta_{n,k_n}^{(2)}J(\omega)=\Delta J_t(\omega)>0,
\end{eqnarray*}
where $k_n(t)=\lceil nt \rceil$  and the equality holds because $J$ has c\`adl\`ag paths. As a consequence, we obtain for $\omega\in \lbrace \omega\in\Omega \mid \sum_{k=1}^n\vert \Delta_{n,k}^{(1)}J(\omega) \vert^p<\infty \rbrace   \cap \Omega_1$, that
\begin{eqnarray*}
\lim_{n\to\infty}(n^{-\gamma}a_n(T_n(\omega)-b_n))\geq C a_n n^{H-\gamma}\liminf_{n\to\infty}\max_{2\leq j\leq n} \Delta_{n,j}^{(2)}J(\omega)+\mathcal{O}_{\mathbb{P}}(1).\hfill\qed
\end{eqnarray*}

\subsection{Proofs of consistent localization\label{localizationproofs}}
\noindent
\textit{Proof of Proposition \ref{proplocalization}:} 
Under the assumptions of Proposition \ref{proplocalization}, $\omega\in\Omega_1$, with one jump at time $\tau\in(0,1)$, and $\Delta J_t=0$, for all $t\in[0,1]\setminus\{\tau\}$. Set $\delta=|\Delta J_{\tau}|$, and $j^*=\lceil \tau n\rceil$, such that
\[|\Delta_{n,j^*}^{(2)} J|=\delta=|\Delta_{n,j^*+1}^{(2)} J|\,,\]
while \(|\Delta_{n,j}^{(2)} J|=0\), for all $j\in\{2,\ldots,n\}\setminus\{j^*,j^*+1\}$.
We use the notation from \eqref{sigmahat} to write
\[\hat\tau_n=\frac{1}{n}\operatorname{argmax}_{2\le j\le n }\frac{|\Delta_{n,j}^{(2)} X|}{\Big(h_n^{-1}C_p^{-1}\sum_{i=\lfloor jh_n^{-1}\rfloor}^{\lfloor jh_n^{-1}\rfloor+h_n}|\Delta_{n,j}^{(2)} X|^p\Big)^{1/p}}=\frac{1}{n}\operatorname{argmax}_{2\le j\le n }\frac{|\Delta_{n,j}^{(2)} X|n^H}{\Big(\widehat{\sigma^p_{\lfloor jh_n^{-1}\rfloor/n}}\Big)^{1/p}}.\]
If it holds true that
\begin{align}\label{condloc}2n^H\,\max_{2\le j\le n }|\Delta_{n,j}^{(2)} Y|\le C\delta n^H\,,\end{align}
with any constant $0<C<\infty$, we obtain for any $j<j^*$ that almost surely
\begin{align*}
\max_{2\le i\le j }\frac{|\Delta_{n,i}^{(2)} X|n^H}{\Big(\widehat{\sigma^p_{\lfloor ih_n^{-1}\rfloor/n}}\Big)^{1/p}}&=\max_{2\le i\le j }\frac{|\Delta_{n,i}^{(2)} Y|n^H}{\Big(\widehat{\sigma^p_{\lfloor ih_n^{-1}\rfloor/n}}\Big)^{1/p}}\\
&\le \max_{2\le i\le n }\frac{|\Delta_{n,i}^{(2)} Y|n^H}{\Big(\widehat{\sigma^p_{\lfloor ih_n^{-1}\rfloor/n}}\Big)^{1/p}}\\
&\le \frac{\delta n^H}{\Big(\widehat{\sigma^p_{\lfloor j^*h_n^{-1}\rfloor/n}}\Big)^{1/p}}-\max_{2\le i\le n }\frac{|\Delta_{n,i}^{(2)} Y|n^H}{\Big(\widehat{\sigma^p_{\lfloor ih_n^{-1}\rfloor/n}}\Big)^{1/p}}\\
&\le \frac{|\Delta_{n,j^*}^{(2)} J| n^H}{\Big(\widehat{\sigma^p_{\lfloor j^*h_n^{-1}\rfloor/n}}\Big)^{1/p}}-\frac{|\Delta_{n,j^*}^{(2)} Y|n^H}{\Big(\widehat{\sigma^p_{\lfloor j^*h_n^{-1}\rfloor/n}}\Big)^{1/p}}\\
&\le \frac{|\Delta_{n,j^*}^{(2)} X| n^H}{\Big(\widehat{\sigma^p_{\lfloor j^*h_n^{-1}\rfloor/n}}\Big)^{1/p}}\,.
\end{align*}
For the second inequality we apply \eqref{condloc} and in the last step the reverse triangle inequality. We use that since $(\sigma_t)$ is bounded from below and above, we have by Lemma \ref{GK} almost sure lower and upper bounds for the block-wise volatility estimates which allows to use \eqref{condloc}. Analogous inequalities show under the condition \eqref{condloc} for any $j>j^*+1$ that almost surely
\begin{align*}
\max_{j\le i\le n }\frac{|\Delta_{n,i}^{(2)} X|n^H}{\Big(\widehat{\sigma^p_{\lfloor ih_n^{-1}\rfloor/n}}\Big)^{1/p}}\le \frac{|\Delta_{n,j^*}^{(2)} X| n^H}{\Big(\widehat{\sigma^p_{\lfloor j^*h_n^{-1}\rfloor/n}}\Big)^{1/p}}\,.
\end{align*}
We can use these bounds under the condition \eqref{condloc} to bound the expected absolute estimation error:
\begin{align*}\E\big[|\hat\tau_n-\tau|\big]&=\E\big[|\hat\tau_n-\tau|\cdot\1(n\hat\tau_n\in\{j^*,j^*+1\})\big]+\E\big[|\hat\tau_n-\tau|\cdot \1(n\hat\tau_n\notin\{j^*,j^*+1\})\big]\\
&\le 2n^{-1}+\P\big(n\hat\tau_n\notin\{j^*,j^*+1\}\big)\\
&\le 2n^{-1}+\P\Big(2n^H\,\max_{2\le j\le n }|\Delta_{n,j}^{(2)} Y|> C\delta n^H\Big)\\
&\le 2n^{-1}+\sum_{j=2}^n\P\Big(2n^H\,|\Delta_{n,j}^{(2)} Y|> C\delta n^H\Big)\\
&=\mathcal{O}\Big(2n^{-1}+\sum_{j=2}^n\exp\Big(-\frac{\delta^2n^{2H}C^2}{2C_H^2\sigma^2_{j/n}}\Big)\Big)\\
&=\mathcal{O}\Big( 2n^{-1}+n\cdot\exp\Big(-\frac{\delta^2n^{2H}C^2}{2C_H^2V^2}\Big)\Big)=\mathcal{O}(n^{-1})\,.
\end{align*}
We use a Gaussian tail bound, that is, for Gaussian random variables exponential bounds of tail probabilities are well-known. The constant $C_H$ is from the variance of second-order increments and $V$ is the lower bound for $(\sigma_t)$ from \eqref{bounded}. The last step holds true, since $\lim_{n\to\infty}n^2\cdot \exp\big(-c \,n^{2H}\big)=0$, with any positive constant $c$. As long as $\delta n^{H}$ increases at some polynomial speed, the resulting order remains valid. We conclude with a standard estimate based on Markov's inequality.

\subsection{Proofs of jump-robust inference for rough processes\label{roughproofs}}
\noindent
\textit{Proof of Proposition \ref{proprobust1}:} We decompose
\begin{align*}n^{2H-1}\sum_{j=1}^n\big(\Delta_{n,j}^{(1)}X\big)^2=n^{2H-1}\sum_{j=1}^n\big(\Delta_{n,j}^{(1)}Y\big)^2+n^{2H-1}\sum_{j=1}^n\big(\Delta_{n,j}^{(1)}J\big)^2+2n^{2H-1}\sum_{j=1}^n\Delta_{n,j}^{(1)}Y\,\Delta_{n,j}^{(1)}J\,.\end{align*}
The central limit theorem given in Theorem 4 from \cite{corcuera2006power} yields that
\begin{align}\label{rob1}n^{2H-1}\sum_{j=1}^n\big(\Delta_{n,j}^{(1)}Y\big)^2=\int_0^1\sigma_s^2\,\intdiff s+\mathcal{O}_{\P} \big(n^{-1/2}\big)\,.\end{align}
Since the jumps have finite quadratic variation, the sum of the squared jumps over any bounded interval is finite, such that
\[n^{2H-1}\sum_{j=1}^n\big(\Delta_{n,j}^{(1)}J\big)^2=\mathcal{O}_{\P}(n^{2H-1})\,.\]
Since the cross term has expectation zero, it suffices to bound its variance with the Cauchy-Schwarz inequality:
\begin{align*}\E\Big[\Big(2n^{2H-1}\sum_{j=1}^n\Delta_{n,j}^{(1)}Y\,\Delta_{n,j}^{(1)}J\Big)^2\Big]& \le 4 n^{4H-2}\sum_{j=1}^n\E\big[\big(\Delta_{n,j}^{(1)}Y\big)^2\big]\,\sum_{j=1}^n\E\big[\big(\Delta_{n,j}^{(1)}J\big)^2\big]\\
&\le 4 \,n^{2H-1}\sum_{j=1}^n\E\big[\big(\Delta_{n,j}^{(1)}Y\big)^2\big]\,n^{2H-1}\sum_{j=1}^n\E\big[\big(\Delta_{n,j}^{(1)}J\big)^2\big]\\
&=\mathcal{O}\big(n^{2H-1} \big)\,,\end{align*}
what readily yields that
\[2n^{2H-1}\sum_{j=1}^n\Delta_{n,j}^{(1)}Y\,\Delta_{n,j}^{(1)}J=\mathcal{O}_{\P}\big(n^{H-1/2} \big)\,.\]
We exploited the assumed independence of $(Y_t)$ and $(J_t)$. If $H\in(0,1/2)$, we conclude \eqref{intvola}.\hfill\qed

\noindent
\textit{Proof of Proposition \ref{proprobust2}:} Based on Proposition \ref{proprobust1} and an analogous estimate for the discrete quadratic variation with distances $2/n$, we obtain for $H\in(0,1/2)$ that
\begin{align*}
\hat H_n&=\frac{1}{2\log(2)}\log\bigg(\frac{2^{2H}\sum_{j=0}^{n-2}\big(X_{(j+2)/n}-X_{j/n}\big)^2(2/n)^{-2H}n^{-1}}{n^{2H-1}\sum_{j=0}^{n-1 }\big(X_{(j+1)/n}-X_{j/n}\big)^2}\bigg)\\
&=\frac{1}{2\log(2)}\log\bigg(2^{2H}\frac{\int_0^1\sigma_s^2\,\intdiff s+\mathcal{O}_{\P}\big(n^{\max((2H-1), -1/2)}\big)}{\int_0^1\sigma_s^2\,\intdiff s+\mathcal{O}_{\P}\big(n^{\max((2H-1), -1/2)}\big)}\bigg)\\
&=\frac{1}{2\log(2)}\log\big(2^{2H}\big)+\mathcal{O}_{\P}\big(n^{\max((2H-1), -1/2)}\big)\\
&=H+\mathcal{O}_{\P}\big(n^{\max((2H-1), -1/2)}\big)\,.
\end{align*}
For the last step, we exploit a bivariate Taylor expansion 
\begin{align*}
\log\Big(\frac{x+\delta_x}{z+\delta_z}\Big)=\log\Big(\frac{x}{z}\Big)+\frac{1}{x}\delta_x-\frac{1}{z}\delta_z+\mathcal{O}\big(\delta_x^2\vee \delta_z^2)\,,
\end{align*}
as $\delta_x,\delta_z\to 0$. Since $\int_0^1\sigma_s^2\,\intdiff s$ is almost surely lower bounded, by Proposition \ref{proprobust1} both power variation statistics which converge to the integrated squared volatility in probability have as well almost surely positive lower bounds, such that we conclude the stochastic order of the remainder.

If  $H\in(1/2,1)$, we obtain that
\begin{align*}
\hat H_n&=\frac{1}{2\log(2)}\log\bigg(\frac{\sum_{j=0}^{n-2}\big(J_{(j+2)/n}-J_{j/n}\big)^2+\mathcal{O}_{\P}\big(n(2/n)^{2H}+n^{1/2}(2/n)^{H}\big)}{\sum_{j=0}^{n-1 }\big(J_{(j+1)/n}-J_{j/n}\big)^2+\mathcal{O}_{\P}\big(n\,n^{-2H}+n^{1/2}\,n^{-H}\big)}\bigg)\\
&=\frac{1}{2\log(2)}\log\bigg(\frac{\sum_{j=0}^{n-2}\big(J_{(j+2)/n}-J_{j/n}\big)^2}{\sum_{j=0}^{n-1 }\big(J_{(j+1)/n}-J_{j/n}\big)^2}\bigg)+\KLEINO_{\P}(1)\,.
\end{align*}
The numerator in the logarithm satisfies
\begin{align*}
\sum_{j=0}^{n-2}\big(J_{(j+2)/n}-J_{j/n}\big)^2&=\sum_{j=0}^{n-2}\big(\big(J_{(j+2)/n}-J_{(j+1)/n}\big)^2+\big(J_{(j+1)/n}-J_{j/n}\big)^2\big)\\
&\hspace*{1cm}+2\sum_{j=0}^{n-2}\big(J_{(j+2)/n}-J_{(j+1)/n}\big)\big(J_{(j+1)/n}-J_{j/n}\big)\\
&=2\sum_{j=0}^{n-1}\big(J_{(j+1)/n}-J_{j/n}\big)^2+\KLEINO_{\P}(1)\,.
\end{align*}
To show that the cross term is asymptotically negligible, we work under the assumption that $(J_t)$ is an It\^{o} semimartingale with bounded jump sizes and of finite quadratic variation. A standard approach in the literature on statistics for semimartingales is to decompose the jumps
\begin{align*}J_t=J_t^{M,q}+(J_t-J_t^{M,q})\,,\end{align*}
into large jumps, $(J_t-J_t^{M,q})$, say of size larger than $1/q$, and compensated small jumps $J_t^{M,q}$, see for instance Remark 3.2 of \cite{vetter} who uses this for a related problem in the analysis of bipower variation statistics in presence of jumps. For any $q<\infty$, $(J_t-J_t^{M,q})$ is a finite-activity jump process which exhibits only finitely many jumps on any fix time interval. Since the probability of the set on that this process exhibits more than one jump within some time interval of length $2/n$ tends to 0 as $n\to\infty$, the cross term due to this component has to be zero. The same argument is detailed in the book by  \cite{jacodprotter}, for instance in Section 13.4.2.
Focusing hence on the small jumps, we can use their martingale structure to conclude. A standard estimate for the jump martingale of an It\^{o} semimartingale with bounded jump sizes yields for any $p\ge 2$ that
\begin{align}\label{mombound}\E\big[\big|J_{(j+1)/n}^{M,q}-J_{j/n}^{M,q}\big|^p\big]=\mathcal{O}\big(n^{-1}e_q\big)\,,\end{align}
uniformly for all $j$, with some constants $e_q$, which satisfy $e_q\to 0$, as $q\to\infty$. This is implied by Lemma 2.1.5 of \cite{jacodprotter}, similar as Eq.\ (9.5.7) in \cite{jacodprotter}. Since martingale increments have expectation zero and are uncorrelated, we obtain that
\[2\sum_{j=0}^{n-2}\big(J_{(j+2)/n}-J_{(j+1)/n}\big)\big(J_{(j+1)/n}-J_{j/n}\big)=\KLEINO_{\P}(1)\,,\]
as the expectation of the term is zero and the variance has an upper bound which tends to zero by an application of Hölder's inequality using the moment bounds \eqref{mombound}.

Finally, consider the case $H=1/2$, when squared jumps and the quadratic variation of the continuous part are balanced and of the same magnitude. In this case, the realized volatility converges to the quadratic variation, including the integrated squared volatility and the sum of the squared jumps. Since the expectation, variance and covariances of increments of the continuous, fractional process imply for $H=1/2$ that
\begin{align*}\sum_{j=0}^{n-2}\big(Y_{(j+2)/n}-Y_{j/n}\big)^2&=\sum_{j=0}^{n-2}\sigma^2_{\frac{j+1}{n}}\,2n^{-1}+\mathcal{O}_{\P}\big(n^{-1/2}\big)\\
&=2\int_0^1\sigma_s^2\intdiff s +\mathcal{O}_{\P}\big(n^{-1/2}\big)\,,
\end{align*}
under Assumption \ref{Annahmen}, while \eqref{rob1} applies to the terms in the denominator, and since we can re-use the above decomposition for the sum of squared jumps and estimates of mixed terms, the ratio in the logarithm converges to 2 again.\hfill\qed 
%
%
\addcontentsline{toc}{section}{References}
\bibliographystyle{apalike}
\bibliography{literature}

\end{document}